\documentclass[12pt,twoside,reqno]{amsart}
  
  \usepackage[OT1]{fontenc}
  \RequirePackage[english]{babel}

  \usepackage{enumerate}
  
  \usepackage{amsthm}
  \RequirePackage{amsmath,amsfonts,amssymb,amsthm}
  
  \RequirePackage[dvips]{graphicx}
  
  \usepackage{cite}
  \usepackage{psfrag}
  \usepackage{verbatim}
  \usepackage[usenames]{color}
  \usepackage[dvips]{graphicx} 
  
  \numberwithin{equation}{section}

  \newcommand{\N}{\mathbb{N}}         
  \newcommand{\R}{\mathbb{R}}         
  \newcommand{\C}{\mathbb{C}}
  \newcommand{\EE}{\mathbb{E}}        
  \newcommand{\PP}{\mathbb{P}}        
  \newcommand{\supp}{\text{supp}}        
  \newcommand{\BB}{\mathcal{B}}         
  \newcommand{\diam}{\text{diam}}       

  \newcommand{\Eucl}{\mathrm{Eucl}}
  \newcommand{\Heis}{\mathrm{Heis}}
  
  \newcommand{\pdc}{\mathbb{P}^2_{\C}}
  \newcommand{\Heisen}{\mathbb{H}}
  \newcommand{\HH}{\mathcal{H}}
  
\newcommand{\Kor}{Kor\'{a}nyi~}
  
  \newcommand{\SB}{\mathbf{S}}
  
  \newcommand{\udimb}{\overline{\mathrm{dim}}_{\mathrm{B}}}
  \newcommand{\dimh}{\mathrm{dim}_{\mathrm{H}}}

  \newcommand{\Chains}{\mathcal{L}_\C}

  \newtheorem*{theoremA}{Theorem A}
  \newtheorem*{theoremB}{Theorem B}
  
  \newtheorem{theorem}{Theorem}[section]
  \newtheorem{lemma}[theorem]{Lemma}
  \newtheorem{prop}[theorem]{Proposition}
  \newtheorem{cor}[theorem]{Corollary}
  
  \newtheorem{remark}[theorem]{Remark}
  \newtheorem{remarks}[theorem]{Remarks}

  \theoremstyle{remark}
  \newtheorem*{defn}{Definition}
  \newtheorem{rem}[theorem]{Remark}

  \DeclareMathOperator{\spt}{spt}

  \DeclareMathOperator{\chsp}{\mathcal{L}_\C}

  \DeclareSymbolFont{bbold}{U}{bbold}{m}{n}
  \DeclareSymbolFontAlphabet{\mathbbold}{bbold}
  \DeclareMathOperator{\Id}{\mathrm{Id}}
  

  \author{Laurent Dufloux}
\address{Department of Mathematics and Statistics, 
PO Box 35
FI-40014 University of Jyv\"askyl\"a}
\email{laurent.s.dufloux@jyu.fi}

\author{Ville Suomala}
\address{Department of Mathematical Sciences, University of Oulu, Finland}
\email{ville.suomala@oulu.fi}
\thanks{This work was partially supported by the Academy of Finland via the Centre of Excellence in Analysis and Dynamics Research.}
\urladdr{http://cc.oulu.fi/~vsuomala/}
  
\begin{document}

\title[Poisson cut-outs in the Heisenberg group and the visual sphere]{Projections of Poisson cut-outs in the Heisenberg group and the visual $3$-sphere}

\begin{abstract}
	We study projectional properties of Poisson cut-out sets $E$ in non-Euclidean spaces. In the first Heisenbeg group $\Heisen=\C\times\R$, endowed with the \Kor metric, we show that the Hausdorff dimension of the vertical projection $\pi(E)$ (projection along the center of $\Heisen$) almost surely equals $\min\{2,\dimh(E)\}$ and that $\pi(E)$ has non-empty interior if $\dimh(E)>2$. As a corollary, this allows us to determine the Hausdorff dimension of $E$ with respect to the Euclidean metric in terms of its Heisenberg Hausdorff dimension $\dimh (E)$.
	
	We also study projections in the one-point compactification of the Heisenberg group, that is, the  $3$-sphere $\SB^3$  endowed with the visual metric $d$ obtained by identifying $\SB^3$ with the boundary of the complex hyperbolic plane. 
	In $\SB^3$, we prove a projection result that holds simultaneously for all radial projections (projections along so called ``chains''). This shows that the Poisson cut-outs in $\SB^3$ satisfy a strong version of the Marstrand's projection theorem, without any exceptional directions.
\end{abstract} 

\maketitle

\section{Introduction}
In this paper, we investigate strong Marstrand-type projection theorems for random cut-out sets in 
  two (related) non-Euclidean spaces: the (first) Heisenberg group $\Heisen$, and its compactification, that is the $3$-sphere $\SB^3$ endowed
  with the visual metric $d$ obtained by identifying $\SB^3$ with the boundary of the complex hyperbolic plane. 
  
  Our focus is on certain projections of these  cut-out sets and their dimension.

  In  the Heisenberg group $\Heisen$, we look at the dimension of the vertical projection (along the center) as well as the dimension of the fibers; as an 
  interesting corollary, this allows us to compute the Hausdorff dimension of the cut-out set with respect to the \emph{Euclidean} metric 
  on $\Heisen$. The following is an informal version of our main theorem in the Heisenberg group.

  \begin{theoremA}
    Let $E$ be a random Poisson cut-out set in the Heisenberg group, with Hausdorff dimension $\beta$. Then with positive probability, 
    \begin{enumerate}
      \item $\pi_Z(E)$, the vertical projection of $E$, has Hausdorff dimension $\inf \{\beta, 2 \}$; and 
      if $\beta > 2$, $\pi_Z(E)$ has non-empty interior;
      \item the Hausdorff dimension of $E$ with respect to the Euclidean metric is equal to 
      \[ \phi(\beta)=\left\{ \begin{array}{ccc} \beta & \mathrm{if} & 0 < \beta \leq 2 \\ 
        2+\frac{1}{2}(\beta-2) & \mathrm{if} & 2 < \beta \leq 4\,. \\
      \end{array} \right. \]
    \end{enumerate}
  \end{theoremA}
Recall that for any subset of the Heisenberg group with Hausdorff dimension $\beta$, the Euclidean Hausdorff dimension is at most 
$\phi(\beta)$ (see e.g. \cite[ Theorem 1.1]{Balogh2003}) so that the random sets $E$ have the maximal Euclidean Hausdorff dimension in terms 
of their Heisenberg dimension.

In the classical Euclidean setting, if $X$ is a random Poisson cut-out set in $\R^n$ with Hausdorff dimension $s \in\ ]0,n[$, then, with 
positive probability, for \emph{any} orthogonal projection $\pi : \R^n \to \R^d$, the image $\pi(X)$ has Hausdorff dimension $\inf\{d, s\}$ \cite{ShmerkinSuomala}. To generalize this 
result to Heisenberg group in a meaningful way, we would need to introduce a family of projections that is a suitable generalization 
of the family of Euclidean projections. One way to do this would be to start from the quotient mapping along the center, $\pi_Z$, considered in 
Theorem A, and to move around the point at infinity. In this paper, we will actually work in the compactification of the Heisenberg group,
that is the $3$-sphere $\SB^3$ endowed with the visual distance that comes from identifying $\SB^3$ with the boundary at infinity of
 the complex hyperbolic plane. The foliation of $\Heisen$ by translates of the center $Z$ yields, in the compactification, a foliation of 
 $\SB^3 \setminus \{ \infty \}$ by the so-called \emph{chains} passing through $\infty$. By moving $\infty$ around $\SB^3$, one 
 obtains the family of projections needed; more precisely, if $x$ is some fixed point of $\SB^3$, any other point $y$ lies on a 
unique chain passing through $x$; this defines the radial projection along chains passing through $x$, or, in short, 
\emph{radial projection at $x$} which can be defined so as to take values in the Euclidean sphere $\SB^2$.

At this point, let us emphasize the following:
\begin{center}
  {\bf Unless stated otherwise, $\SB^3$ will always be endowed with the visual metric $d$ coming from the identification
  with the visual boundary of the complex hyperbolic plane.}
\end{center}
This is \emph{not} the same thing as the visual metric coming from the identification with the visual boundary of the real hyperbolic $4$-space. 
The former has dimension $4$ whereas the latter is the familiar Euclidean $3$-sphere and has dimension $3$.

Given a random Poisson cut-out set $E \subset \SB^3$, we can, with positive probability, compute the Hausdorff dimension of the image of 
$E$ through the radial projections \emph{at every point $x \in \SB^3$ simultaneously}.

Thus, our work is related to the recent program  aiming to show that for many sets and measures of random or dynamical origin,
 the statement of the Marstrand's projection theorem holds without any ``exceptional'' directions. 
 See e.g. \cite{ShmerkinHochman, RamsSimon, ShmerkinSuomala} and references therein. The following is our main result.

  \begin{theoremB}
    Let $E$ be a random Poisson cut-out set in  $\SB^3$ (endowed with the visual distance $d$), with Hausdorff dimension $\beta \in\ ]0, 4[$. 
    Then with positive probability, for every point $x$ of $\SB^3$, the radial projection of 
    $E$ at $x$ has Hausdorff dimension $\inf\{2,\beta\}$, and non-empty interior if $\beta > 2$.
  \end{theoremB}

We refer the reader to Section \ref{s.grom} for the exact definition of the radial projection, the
 definition of the visual metric on $\SB^3$ as well as the Poisson cut-out sets we consider. 

 We note that the ``radial projections''  we consider are also studied in \cite{Dufloux2017} where a Marstrand-type projection result is 
 stated: if $A$ is a Borel subset of $\SB^3$ of Hausdorff dimension $\alpha$ \emph{with respect to the Euclidean metric $d_E$}, then
 for Lebesgue-\emph{almost every} $x \in \SB^3$, the radial projection of $A$ at $x$  has Hausdorff dimension $\inf\{2,\alpha\}$. This is 
 a special case of Theorem 5 in \cite{Dufloux2017}; pay attention to the fact that in this result the dimension of $A$ is 
 computed with respect to the Euclidean metric. In fact, this result is not true if we consider the visual metric $d$ instead. For instance, the chains in $\SB^3$ are $2$-dimensional, but their radial projections always have Hausdorff dimension $1$, see Remark \ref{rem:marstrand_fails}. Nevertheless, our main results shows that the behaviour of random sets under the radial projections resembles that of a strong Marstrand theorem: with positive probability, the dimension of the projection takes the ``expected value" simultaneously for all projections.

 Many authors have previously studied Marstrand-type projection theorems in the Heisenberg group, see e.g.  \cite{BaloghProjections}; the projections
 studied by these authors are quite different in nature. Namely, they consider projections \emph{onto} horizontal homogeneous subgroups of $\Heisen$, 
 \emph{i.e.} subgroups of the form $V_\theta = e^{i\theta} \mathbf{R} \times \{0\} \subset \C \times \R$, as well as projections \emph{along} these subgroups.
 The projection onto $V_\theta$ is essentially the same thing as the vertical projection $\pi_Z$ followed by an orthogonal projection in $\C$, and there
 is not much to say beyond Marstrand's original Theorem in the plane. Projections \emph{along} the horizontal $V_\theta$ are more interesting but 
 also very different from the projections along chains we are considering. In fact the $V_\theta$ and their translates  are the $\R$-circle passing 
 through the point at infinity, they are in some way the opposite of the chains we are looking at. In the boundary of complex hyperbolic plane, 
 a chain is the boundary of a totally geodesic complex submanifold, of sectional curvature $-4$, whereas a $\R$-circle is the boundary of a
 totally geodesic real $2$-submanifold of sectional curvature $-1$. We refer to \cite{Goldman} for these notions.

 The main ingredient of the proofs of our main theorems is an abstract result, Theorem \ref{th:holdcont},
  which holds for Poisson cut-out sets under fairly general hypotheses. The result is a straightforward generalization
   of the main result in \cite{ShmerkinSuomala} into a non-Euclidean setting. In order to apply Theorem \ref{th:holdcont} in our 
   non-Euclidean setting, we need to derive a geometric estimate of H\"{o}lder type; in the Heisenberg group, this boils down to 
   estimating the intersections of vertical lines with Heisenberg balls, see  \eqref{eq.holdest1}. 
   In $\SB^3$ the corresponding estimate is somewhat more involved (see Lemma \ref{l.technical}).

  The paper is organized as follows. In Section \ref{sec2}, we recall the construction of random Poisson cut-out sets (and measures) and the formula giving their Hausdorff dimension with positive probability. In Section \ref{sec3}, we state the H\"{o}lder regularity Theorem \ref{th:holdcont}; this result will allow us to control  
  the measure of cut-out sets along families of ``fibres''; we also state some elementary Lemmas that allow us to derive dimensionality results from the regularity provided by the Theorem. 
  These results are applied in Section \ref{s.heis} where we deal with the Heisenberg group and its non-Euclidean metric; this is where Theorem A is proved. This section is also a warm-up for the 
  next one, which is more technical and deals with $\SB^3$ endowed with the visual metric $d$.
  We spend some time introducing 
  the needed properties of this metric, and defining the family of projections we are studying. The main argument for the proof of Theorem~B is in Section \ref{ss.mainproof},
  and the most technical part
  (where we prove the geometric H\"{o}lder estimate needed to apply Theorem \ref{th:holdcont}) is deferred to Section \ref{ss.tech}.

  Hausdorff dimension of sets will be denoted by $\dimh$; Hausdorff dimension of measures will be denoted by $\dim$. Recall that by definition 
  \[ \dim(\mu) = \inf\{ \dimh(A)\ ;\ \mu(A) > 0\}\,. \]
The upper box dimension will be denoted by $\udimb$. The $s$-dimensional Hausdorff measure is $\mathrm{H}^s$. For definitions, see \cite{Mattila}.
  
  The closed, resp. open, ball of radius $r$ and center $x$ is denoted by $B(x,r)$, resp. $B^{\circ}(x,r)$.
  
Positive and finite constants will be denoted by $c,C$, etc. When there is no danger of misunderstanding,
 we are quite flexible in the notation, for instance, the value of $C$ may change from line to line. We will use subscript,
  when there is a need to stress the dependency of a constant on certain parameters. For instance, $C_\varepsilon$ is a 
  positive and finite constant whose value may depend on a parameter $\varepsilon>0$ (but not on other variables relevant for the context). 
  If $0<A,B<\infty$ are variables and $A\le C B$, we will denote $A\lesssim B$. The notation $B\asymp A$ means that $A\lesssim B$ and $B\lesssim A$. When necessary, the dependency
  will be indicated with a subscript notation, \emph{i.e.} if $A \leq C B$ where $C$ depends on some data $D$, we will write $A \lesssim_D B$.

\section{Dimension of conformal Poisson cut-out sets}\label{sec2}

In this Section, we define random Poisson cut-out sets and measures and recall some results regarding their Hausdorff and box-counting dimensions. We present these results in a generality that is sufficient for the purposes of the paper. For more general results, see e.g. \cite{Mandelbrot72, Zahle85, Ojalaetal}.

  Let $\mathcal{Z}_0$ be a boundedly compact metric space 
  and assume that for some $m>0$ it carries a Borel measure $\HH$  
  such that: For any $x \in\mathcal{Z}_0$,
  \begin{equation}\label{eq:m-unif}
   \HH(B(x,r))=f(r)\,, 
   \end{equation}
   where
   \begin{equation}\label{eq:m-lim}
   \lim_{r\to 0}\frac{f(r)}{r^m}=1\,.
   \end{equation}

Later in this paper, we will only consider the case $m=4$. More precisely, $\mathcal{Z}_0$ will be either the Heisenberg group $\Heisen$,
   or its compactification $\SB^3$ endowed with the visual metric, and $\HH$ will be a suitable normalization of the usual Lebesgue measure (resp. surface measure) on $\Heisen$ (resp. $\SB^3$).

  We endow $\mathcal X = \mathcal{Z}_0 \times ]0,1]$ with the $\sigma$-finite measure
  \[ \mathbf{Q} = \HH \otimes \frac{\mathrm{d}r}{r^{m+1}} \mathbf{1}_{0 < r \leq 1 }\text. \]
 To any pair $(x,r) \in \mathcal X$ we associate the \emph{closed} ball $B(x,r)\subset \mathcal{Z}_0$.

  For any real number $\gamma > 0$, we consider a Poisson point process of intensity $\gamma \mathbf{Q}$ on $\mathcal X$. 
  For convenience of the reader, let us recall the definition.

  \begin{defn}
    Let $\mathcal X$ be a complete separable metric space
    and let $\mathbf{Q}$ be a $\sigma$-finite Borel measure on $\mathcal{X}$. A \emph{Poisson point process with intensity}
    $\mathbf{Q}$ is a random subset $\mathcal Y \subset \mathcal X$ such that
    \begin{itemize}
      \item For each Borel set $\mathcal{A}\subset X$, the number
      $N(\mathcal{A}):=\#\mathcal{A}\cap\mathcal{Y}$ is a Poisson
      random variable with mean $\mathbf{Q}(\mathcal{A})$.
      \item For pairwise disjoint Borel sets
      $\mathcal{A}_i\subset\mathcal{X}$, $i\in\N$, the random variables
      $N(\mathcal{A}_i)$ are independent.
    \end{itemize}
  \end{defn}
  It is well-known and easy to see that this is a well-defined object.

Returning to our setup ($\mathcal{X}$ endowed with $\gamma \mathbf{Q}$), we let $E^0$ be the associated random Poisson cut-out set:
  \[ E^0 = \mathcal{Z}_0 \setminus \bigcup_{i \in I} B^\circ(x_i,r_i)\,, \]
  where $\mathcal Y = \{(x_i,r_i)\ ;\ i \in I\}$ is the Poisson point process considered. In that setting, the most basic result is the following.
  \begin{prop} \label{pr.basic}
    If $\gamma > m$, then $E^0$ is a.s. empty. If $0 < \gamma \leq m$, then for any bounded  subset $\mathcal{Z}$ of $\mathcal{Z}_0$, almost surely
    \[ \udimb(\mathcal{Z} \cap E^0) \leq m - \gamma \text. \]
    In particular,
    \[ \dimh(E) \leq m- \gamma \text. \]
  \end{prop}

  \begin{proof}
    The proposition is well known, but let us provide the simple proof for reader's convenience. Let $\gamma'<\gamma$ and pick $r_0=r_0(\gamma)>0$ such that 
    \begin{equation}\label{fr}
    f(r)>\frac{\gamma'}{\gamma} r^m
    \end{equation} 
    for $0<r<r_0$. 

    First we bound the probability that some small ball is not eaten out by the cut-out. Let $A = B(x,\delta)$ where $x \in \mathcal{Z}$ and $\delta >0$. Then
    \begin{equation} \PP(A \cap E \neq \varnothing) \leq C \delta^{\gamma'}\,, \label{eq.estimate1}
    \end{equation}
    where $C$ is some constant (which depends on $d$ and $\gamma$). Indeed, in order to cut out the $\delta$-ball $A$ it is enough that there is a ball $B(x_i,r_i)$ such that $r_i > \delta$ and $x$ belongs to $B(x_i,r_i-\delta)$. Now by \eqref{eq:m-unif}, \eqref{fr}, and the definition of $\mathbf{Q}$,
    \begin{align*} 
    \PP \left( x \in \bigcup_{r_i > \delta} B(x_i,r_i-\delta) \right) &=1-\exp\left(-\gamma\int_{r=\delta}^1 \frac{f(r-\delta)}{r^{m+1}}\,\mathrm{d}r\right) 
    \\
    &\ge 1 - \exp \left(-\gamma' \int_\delta^{r_0} \frac{(r-\delta)^m}{r^{m+1}} \mathrm{d} r \right)\\
    &\geq 1-\exp\left(-\gamma'\int_{\delta}^{r_0}\frac{\mathrm{d}r}{r}\right)\\
    &\ge 1-C\delta^{\gamma'}\,,
    \end{align*}
where $C=C_{\gamma'}>0$ is a constant.

    Now for each $n \geq 1$, let $\mathcal Q_n$ be a covering of $\mathcal{Z}$ with balls of radius $2^{-n}$ centered in $\mathcal{Z}$, such that $\# \mathcal Q_n \leq C 2^{nm}$, where $C$ is some fixed constant. It is easy to check that such a $\mathcal Q_n$ does exist for any $n$. Let $N_n$ be the number of $A \in \mathcal Q_n$ that meet $E^0$. We know by the previous computation that $\EE[N_n] \leq C 2^{n(m-\gamma')}$. Thus, for any $\varepsilon>0$,
    \[\EE\left[\sum_{n=1}^\infty 2^{n(\gamma-m-\varepsilon)}N_n \right]<\infty\,.\]
In particular, a.s., $N_n\le 2^{n(m-\gamma+\varepsilon)}$ when $n$ is large. The claims follow from this at once.
  \end{proof}

  We now fix for simplicity a bounded closed subset $\mathcal{Z} \subset \mathcal{Z}_0$ of positive $\HH$-measure.
  There is no hope to prove that the estimate for $\udimb(\mathcal{Z}\cap E^0)$ is almost surely an equality, since, as one may check, the cut-out set $E^0\cap\mathcal{Z}$ is empty with positive probability for any $\gamma>0$.
  On the other hand, it is possible to show that equality holds with positive probability. Unsurprisingly, the proof relies on the construction of a ``natural'' measure on the cut-out set.

  For every $n\in\N$ let
  \[ E_n = \mathcal{Z} \setminus \bigcup_{r_i \geq 2^{-n}} B(x_i,r_i) \]
  and
  \begin{equation}\label{eq:mu_n}
  \mu_n = \beta_n\textbf{1}_{E_n \cap \mathcal{Z}}\,,\end{equation}
where
\[\beta_n=\exp\left(\gamma\int_{r={2^{-n}}}^1\frac{f(r)}{r^{m+1}}\,\mathrm{d}r\right)\]
is the reciprocal of $\PP(x\in E_n)$ (note that $\beta_n$ is independent of $x$). Recall that $\beta_n\sim 2^{\gamma n}$ in the sense that $\lim_{r\downarrow 0}\tfrac{\log{\beta_n}}{n}=\gamma$.
  Let
  \[E=\bigcap_n E_n=E^0\cap\mathcal{Z}\,.\]
  It is easy to see (see e.g. \cite{ShmerkinSuomalasurvey}) that, almost surely, the sequence of (random) finite Radon measures 
  $\mu_n \HH$ converges in the weak*-sense to a finite measure $\mu$ supported on $E$. The following proposition shows that there is an equality in Proposition \ref{pr.basic}
   with positive probability. For a proof, see e.g. \cite{Ojalaetal}. See also Lemma \ref{lem:2moment}.
 
  \begin{prop} \label{pr.dimgen}
    Assume that $\gamma \in ]0,m[$. Then there is a positive probability that $\mu \neq 0$; and, conditional on $\mu \neq 0$, it holds almost surely that $\mu$ is exact-dimensional and 
    \[ \dim(\mu) = \dimh(E) = m -\gamma.\]
  \end{prop}
Recall that exact dimensionality means that the limit
\[\lim_{r\downarrow 0}\frac{\log\left(\mu(B(x,r))\right)}{\log r}\]
exists and obtains a constant value for $\mu$-almost every $x$.

\begin{remark}
Throughout the paper, we will denote by $\PP$ the law of the Poisson point process considered (this depends only on $\mathbf{Q}=\mathbf{Q}(\mathcal{H},\gamma)$) and (with a slight abuse of notation), we will think of $\PP$ as a probability measure on the space of compact subsets of $\mathcal Z$.
\end{remark}

\section{Spatially independent martingales in metric spaces}\label{sec3}

  In this section, we recall a version of the main result of \cite{ShmerkinSuomala} on spatially independent martingales. 
  This will allow us to control the measure of our Poisson cut-out set 
  along the fibres of the projections we will be considering. In the Heisenberg group (Section \ref{s.heis}), we only look at the vertical projection, so the fibres will be the vertical lines. 
  In $\SB^3$ (Section \ref{s.grom})
  we consider the family of radial projections at every point, so the fibres will be all complex chains. (For technical reasons we will have to look at compact spaces of complex chains.)

  Controlling the measure of our Poisson cut-out set along the fibres of the projection is how we will be able to derive results on the projected measure.

  We now describe the abstract setting of the Theorem. 
  Let $\mathcal{Z}$ be a separable 
  locally compact metric space. We consider a random sequence of functions $\mu_n\colon \mathcal{Z}\to[0,+\infty)$, jointly defined on some probability space  enjoying the following properties:

  \begin{itemize}
    \item $\mu_0$ is a deterministic function with bounded support (we will denote its support by $\Omega$).
    \item There exists an increasing filtration of $\sigma$-algebras $\mathcal{B}_n\subset\mathcal{B}$, such that $\mu_n$ is $\mathcal{B}_n$-measurable.
    Moreover, for all $x\in\mathcal{Z}$ and all $n\in\N$,
    \[
    \EE(\mu_{n+1}(x)|\BB_n) = \mu_n(x)\,.
    \]
    \item
    There is $C<\infty$ such that $\mu_{n+1}(x)\le C\mu_n(x)$ for all $x\in\mathcal{Z}$ and $n\in\N$.
    
    \item There is $C<\infty$ such that for any $(C 2^{-n})$-separated family $\mathcal{Q}$ of Borel sets of diameter $\le C^{-1} 2^{-n}$, the restrictions $\{\mu_{n+1}|_Q | \mathcal{B}_n\}$ are independent.
  \end{itemize}

  \begin{defn}
    Following \cite{ShmerkinSuomala}, we call a random sequence $( \mu_n)$ satisfying the above conditions  an \emph{SI-martingale}, (where SI stands for \emph{spatially independent}).
  \end{defn}

  \begin{rem}
    The sequence \eqref{eq:mu_n} is an obvious example of an SI-martingale, and in fact, the only example dealt with in this paper. Note that the dyadic discretization  ($\mu_n$ and $E_n$ are approximations of $\mu$ and $E$ at level $2^{-n}$)  is used for the simplicity of notation only. 
 \end{rem}

  \begin{theorem}[Regularity of fibres] \label{th:holdcont}
    Let $(\mu_n)_{n\in\N}$ be an SI-martingale, and let $(\eta_t)_{t\in\Gamma}$ be a family of finite Radon measures (``fibre measures'') indexed by a metric space $(\Gamma,d)$.
    We assume that there are constants $0<\gamma,\kappa,\theta,\gamma_0,C<\infty$ such that the following holds:
    \begin{enumerate}
      \renewcommand{\labelenumi}{(A\arabic{enumi})}
      \renewcommand{\theenumi}{A\arabic{enumi}}
      \item \label{H:hyp1}
      $\udimb\Gamma<\infty$.
      \item \label{H:hyp2}
      $\eta_t(B(x,r))\le C r^\kappa$ for all $x\in \mathcal{Z}$, $r>0$ and $t\in\Gamma$.
      \item \label{H:hyp3} Almost surely, $\mu_n(x)\le C\, 2^{\gamma n}$ for all $n\in \N$ and $x\in\mathcal{Z}$.
      \item \label{H:hyp4} Almost surely, there is a random integer $N_0$, such that
      \begin{equation}\label{eq:hold}
        \sup_{t,u\in\Gamma,t\neq u;n\ge N_0} \frac{\left|\int\mu_n\,\mathrm{d}\eta_t-\int\mu_n\,\mathrm{d}\eta_u\right|}{2^{\theta n}\,d(t,u)^{\gamma_0}} \le C\,.
      \end{equation}
    \end{enumerate}
    
    Suppose that $\kappa>\gamma$. Then, almost surely, 
    \begin{itemize}
      \item For all $t$, $\int\mu_n\,\mathrm{d}\eta_t$ converges uniformly to a finite number $X(t)$;
      \item For each $t\in\Gamma$ such that $\int \mu_0(x)\,\mathrm{d}\eta_t(x)>0$, we have $\PP(X(t)>0)>0$.
      \item The function $t\mapsto X(t)$ is (H\"older) continuous.
    \end{itemize}
    
    Suppose that $\kappa\le\gamma$. Then, almost surely,
    \[\sup_{n\in\N,\,t\in\Gamma}2^{-\theta n}\int\mu_n\,\mathrm{d}\eta_t<\infty\,,\]
    as long as $\theta>\gamma-\kappa$.
  \end{theorem}

  \begin{remarks}\label{rem:cubes}
    \begin{enumerate}
      \item
      In \cite{ShmerkinSuomala}, the Theorem is stated in the Euclidean setting $\mathcal{Z}=\R^d$ (see Theorems 4.1 and 4.4. in \cite{ShmerkinSuomala}). However, the proofs work verbatim in any metric space $\mathcal{Z}$. The only minor change is in the proof of \cite[Lemma 3.4]{ShmerkinSuomala}, where instead of the dyadic cubes of sizes $2^{-n}$, one should consider a disjoint cover of $\spt\eta_t$ with sets $Q_j$ satisfying $\diam(Q_j)\le C 2^{-n}$ and such that each $Q_j$ contains a ball $B(x_j,2^{-n})$ for some $x_j\in\spt\eta$.
      
      \item As explained in \cite{ShmerkinSuomala}, there is a scope for weakening the assumptions of Theorem \ref{th:holdcont} 
      We shall not discuss these generalizations here
      since the above version is enough for our application in the Heisenberg group and the visual sphere.
      
      \item The H\"older exponent of $t\mapsto X(t)$ is deterministic and quantitative in terms of the data $(\kappa,\gamma,\gamma_0,\theta)$, see \cite{ShmerkinSuomala}.
    \end{enumerate}
  \end{remarks}

  In applying Theorem \ref{th:holdcont} we will need two companion results, Lemmas \ref{l.supcrit} and \ref{l.subcrit}, corresponding to the two possible conclusions in the Theorem.

  \begin{lemma}\label{l.supcrit}
    Let $\mathcal Z$ be a compact metric space endowed with a Radon measure $\mathcal H$. Let also 
    $\pi : \mathcal Z \to \R^k$ be a Lipschitz mapping and, for any $t \in \R^k$, $\eta_t$ be a finite Radon measure supported 
    on $\pi^{-1}(t)$ such that $\mathcal H$ is equivalent to the finite Borel measure
    \begin{equation}\label{equiv_slices}
     \int \eta_t\ \mathrm{d}t  : A \mapsto \int_{\R^k} \eta_t(A)\ \mathrm{d}t 
     \end{equation}
    with Radon-Nikodym derivative uniformly bounded away from $0$ and $+\infty$.
    Finally, let $\mu_n$ be a sequence of bounded Borel functions $\mathcal Z \to [0,\infty[$ such that 
    \begin{enumerate}
      \item The sequence of Radon measures $(\mu_n \mathcal H)_n$ weak*-converges to a finite Radon measure $\mu$;\label{weak_c}
      \item For any $t$, $\int \mu_n \mathrm{d} \eta_t$ converges to a finite number $X(t)$, and the convergence is uniform in $t$;\label{unif_c}
      \item The mapping $t \mapsto X(t)$ is continuous, and there is some $t_0 \in \R^k$ such that $X(t_0) \neq 0$.\label{cont_c}
    \end{enumerate}
    Then the push-forward measure $\pi \mu$ is absolutely continuous, and $t_0$ is an interior point of $\pi (\supp(\mu))$.
  \end{lemma}
\begin{proof}
Using \eqref{equiv_slices} and  \eqref{weak_c}, we get the follwing estimates for the projected measures of balls centered at $u\in\R^k$:
\begin{align*}
\pi\mu(B^\circ(u,r))&\le\liminf_{n\to\infty}\pi(\mu_n\mathcal{H})(B^\circ(u,r))\lesssim\liminf_{n\to\infty}\int_{B^\circ(u,r)}\int\mu_n\,\mathrm{d}\eta_t\,\mathrm{d}t\,,\\
\pi\mu(B(u,r))&\ge\limsup_{n\to\infty}\pi(\mu_n\mathcal{H})(B(u,r))\gtrsim\limsup_{n\to\infty}\int_{B(u,r)}\int\mu_n\,\mathrm{d}\eta_t\,\mathrm{d}t\,.
\end{align*}
Taking \eqref{unif_c} and \eqref{cont_c} into account, it then follows that $\pi\mu$ is absolutely continuous (with respect to the Lebesgue measure on $\R^k$), and that the Radon-Nikodym derivative 
of $\pi\mu$ at $u\in\R^k$ is comparable to $X(u)$. Thus the claim.
\end{proof}

  \begin{lemma}\label{l.subcrit}
In the setting of Lemma \ref{l.supcrit}, suppose that \eqref{equiv_slices} and \eqref{weak_c} hold with $\mu\neq 0$. Assume further, that for some constants $\theta$ and $C$,
\begin{enumerate}
\setcounter{enumi}{3}
\item\label{eq:small_dim}  $\sup_{n\in\N,\,t\in\R^k}2^{-\theta n}\int\mu_n\,d\eta_t<\infty$; 
\item\label{eq:uniform_tube_bound} For each $n\in\N$, there is a $2^{-n}$-dense family $\mathcal{D}_n\subset \pi(\Omega)$ such that
\[\pi\mu\left(B(t,2^{-n})\right)\le C\left(\pi\mu_n\left(B(t,C 2^{-n})\right)+2^{n(\theta-k)}\right)\text{ for all }t\in\mathcal{D}_n, n\in\N\,.\]
\end{enumerate}
Then, $\dim\pi E\ge\dim\pi\mu\ge k-\theta$.
\end{lemma}

\begin{proof}
The assumptions readily imply that if $t\in\mathcal{D}_n$ and $0<r\le 2^{-n}$, then
\begin{equation}\label{dim_est_for_D_n}
\begin{split}
\pi\mu(B(t,r))&\le C\pi\mu_n(B(t,C 2^{-n}))+C 2^{n(\theta-k)}\\
&\le C 2^{n(\theta-k)}+C\int_{t\in B(t,C 2^{-n})}\int\mu_n\,\mathrm{d}\eta_t\,\mathrm{d}t\\
&\le C 2^{n(\theta-k)}\,,
\end{split}
\end{equation}
with constants that are independent of $t,r$ and $n$. Since for an arbitrary $t\in\pi(\Omega)$, $B(t,r)$ may be covered by boundedly many $B(t_i,2^{-n})$, $t_i\in\mathcal{D}_n$, the estimate \eqref{dim_est_for_D_n} continues to hold (with slightly bigger constants), for all $t\in\R^k$, $n\in\N$.
In particular, this means that
$\dim(\pi\mu,t)\ge k-\theta$ for all $t\in\spt\pi\mu$.
\end{proof}

In order to apply Lemma \ref{l.subcrit}, we will also need the following probabilistic statement concerning the convergence speed of the $\mu_n$ measures of a fixed subset of $\mathcal Z$. We state the lemma for measures satisfying \eqref{eq:m-unif}--\eqref{eq:m-lim} although it clearly holds under much more general assumptions.

  \begin{lemma}\label{mu_m_to_mu_c.}
    Let $\mu_n$ be an SI-martingale on a space $\mathcal{Z}$ and let $\mathcal{H}$ be a  measure on $\mathcal{Z}$ satisfying \eqref{eq:m-unif}--\eqref{eq:m-lim}. Suppose $\mu_n(x)\le C 2^{\gamma n}$ for all $x\in\mathcal{Z}$, $n\in\N$. Let $Z\subset\mathcal Z$ be open and $\varrho<m-\gamma$. Then
    \[\PP\left((\mu(Z)>4\left(\mu_n(Z)+2^{-n\varrho}\right)\,|\,\mathcal{B}_n\right)\le C\exp\left(-c2^{n(m-\gamma-\varrho)}\right)\,.\]
  \end{lemma}

\begin{proof}
Applying \cite[Lemma 3.4]{ShmerkinSuomala} with $\eta=\mathcal{H}|T$ 
 and $\kappa_l=(l-n)^{-2}2^{-n\varrho/2}$ yields
\[\PP\left(\mu_{l}(T)\ge\mu_{l-1}(T)+(l-n)^{-2}2^{-n\varrho/2}\sqrt{\mu_{l-1}(T)}\right)\le C\exp\left(-c (l-n)^{-4} 2^{l(m-\gamma)-n\varrho}\right)\,.\]
Noting the bounds $2^{-n\varrho/2}\sqrt{\mu_{l-1}(T)}\le\max\{2^{-n\varrho},\mu_{l-1}(T)\}$ and $\sum_{l>n}(l-n)^{-2}<4$ and summing over all $l>n$ implies
\[\PP\left(\limsup_{l\to\infty}\mu_l(T)\ge 4\left(\mu_n(T)+2^{-n\varrho}\right)\right)\le C\exp\left(-c 2^{n(m-\gamma-\varrho)}\right)\,.\]
Since $T$ is open, $\mu(T)\le\liminf_l\mu_l(T)\le\limsup_l\mu_l(T)$ and the claim follows.
\end{proof}

\section{Conformal cut-outs in the Heisenberg space}\label{s.heis}

  \subsection{Basic facts about the Heisenberg group}\label{ss.bas}
  Let $\mathbb{H}$ denote the Heisenberg group $\C \times\R$ equipped
  with the group law
  $(u,s)\cdot(v,t)=(u+v,s+t+\mathrm{Im}(\bar{u} v))$ and the
  \Kor metric $d(p,q)=||q^{-1}\cdot p||$, where
  $||(u,s)||=(|u|^4+4s^2)^{1/4}$ (here $|u|$ is the usual modulus of $u \in \C$). This is a  boundedly compact separable metric space.

  With this metric, $\mathbb{H}$ has Hausdorff dimension
  $4$. Indeed,  the Haar measure $\mathcal{H}$ on $\mathbb{H}$ (which is just the Lebesgue measure on $\R^3$), suitably normalized, is $4$-uniform, that is,
  \[\mathcal{H}(B(x,r))=r^4\text{ for all }x\in\mathbb{H},\,r>0\,.\]

  The identification $\mathbb{H} = \C \times \R$ allows to endow this space with the usual Euclidean
  metric $d_{\mathrm{E}}$. The following well known lemma describes the way both metrics
  relate to each other. Recall that the center $Z$ of $\Heisen$ is the
  ``vertical'' line $\{0\} \times \R$ and it is also equal to the
  derived group $D(\Heisen)$; we denote by $\pi$ the quotient mapping
  $\Heisen \to \Heisen/Z$.
  \begin{lemma}\label{l.metcomp}
    \begin{enumerate}
      \item The identity mapping from any compact subset of $\Heisen$ into $\C \times \R$ is Lipschitz. \label{trans_comb}
      \item \label{i.vcomp} If $\pi(x)=\pi(y)$, 
      \[ d(x,y) = \sqrt{2} d_{\mathrm{E}}(x,y)^{\frac{1}{2}}\,.\]
      \item \label{i.qcomp} The Euclidean and Heisenberg metrics are equal modulo $Z$,
      \emph{i.e.} for all $u,v \in \Heisen/Z$,
      \[ \inf_{x,y} d_{\mathrm{E}}(x,y) = \inf_{x,y} d(x,y) \] where
      $x$, resp. $y$, runs through $\pi^{-1}(u)$,       resp. $\pi^{-1}(v)$.
    \end{enumerate}
  \end{lemma}
  Note that the identity mapping from $\Heisen$ into $\R^3$ is not globally Lipschitz. Another way to put the third statement is to say that $\Heisen/Z$, endowed with the quotient of the 
  Heisenberg metric, identifies isometrically with the Euclidean plane.

  \subsection{Poisson cut-out sets in Heisenberg group}
  As in section \ref{sec2}, we define the intensity measure
  \[ \mathbf{Q} = \HH \otimes \frac{\mathrm{d}r}{r^5}\mathbf{1}_{r < 1} \]
  on $\Heisen \times ]0,1[$.

  Let $\Omega$ be the unit ball in $\Heisen$. Fix some parameter $\gamma \in ]0,4[$ and consider a random Poisson point process $ \{ (x_i,r_i)\ ;\ i \in I \} \subset \Heisen \times ]0,1[$
  with intensity $\gamma \mathbf{Q}$. The resulting random cut-out set is
  \[ E = \Omega \setminus \bigcup_{i \in I} B(x_i,r_i) \text. \]
  As before, we let also $\mu$ be the random cut-out measure supported on $E$.

  We will denote Hausdorff dimension (of sets and measures) with respect
  to the  \Kor  metric by $\dimh^\Heis$. 

  In this setting, the general Proposition \ref{pr.dimgen} implies the following
  \begin{prop}\label{p.genheis}
    Almost surely, conditional on $\mu \neq 0$,
    \[ \dim^\Heis(\mu) = \dimh^\Heis (E) =  4 - \gamma\,. \]
  \end{prop}

  In what follows, we denote the ``expected Hausdorff dimension" of $E$ (with respect to the Heisenberg metric) by $\beta=4-\gamma$.

  \subsection{Vertical projection of Poisson cutouts in  Heisenberg group}

  \begin{theorem}\label{p.supcrit}
    Almost surely, conditional on $\mu \neq 0$,
    \begin{enumerate}
      \item\label{v1} if $\beta > 2$, the push-forward measure $\pi \mu$ is absolutely continuous and $\pi(E)$ has non-empty interior;
      \item\label{v2} if $\beta\leq 2$, $\dim (\pi \mu) = \dimh (\pi (E)) = \beta$.
    \end{enumerate}
  \end{theorem}

\begin{remark}
The first results concerning the projections of random sets were obtained by Falconer \cite{Falconer1989}, and Falconer and Grimmett \cite{FG92}. According to these results, the vertical projection of a random Cantor set $E\subset\R^2$ has Hausdorff dimension $\min\{1,\dimh(E)\}$, and nonempty interior if $\dimh(E)>1$. The Theorem \ref{p.supcrit} can be considered an analogue of this classical result in the Heisenberg setting.
\end{remark}

Before proving Theorem \ref{p.supcrit}, let us state the main geometric ingredient of the proof. For any $u \in \Heisen/Z$, let 
    $\eta_u$ be the $2$-dimensional Hausdorff measure on $\pi^{-1}(u)$, $\eta_u = \mathrm{H}^2|\pi^{-1}(u)$; the reader may check that $\eta_u$ is equal to the usual Lebesgue measure on the affine line $\pi^{-1}(u)$.

\begin{lemma}\label{l.h_cont}
There is a constant $0<C<\infty$ such that for all $0<r\le 1$ and all $u,v\in\Heisen/Z$,    
\begin{equation}\label{eq.holdest1}
      \left| \eta_u (B(0,r)) - \eta_v(B(0,r)) \right| \le C |u-v|^{\frac12}\,. 
    \end{equation}
\end{lemma}

\begin{proof}
By definition, 
    \[B(0,r)=\{(z,t)\in \Heisen\,:\,|z|^4+t^2\le r^4\}\,,\]
    and a straightforward computation (using the fact that $\eta_u$ is the Lebesgue measure on $\pi^{-1}(u)$, and assertion \ref{i.vcomp} in 
    Lemma \ref{l.metcomp}) gives
    \[
    \mathrm{H}^2(L_u \cap B(0,r)) = \left\{ \begin{array}{ccc} c_1 \sqrt{r^4-|u|^4} & \mathrm{if} & |u| \leq r \\ 0 & \mathrm{otherwise} &
    \end{array}
    \right.\,,\]
    (where $c_1$ is some constant) 
    which after a simple computation leads to the estimate
    \[\mathrm{H}^2(L_u\cap B)-\mathrm{H}^2(L_v\cap B)|\le c_2\sqrt{|u-v|}\,.\]
    (Here we are ignoring the term $\sqrt{r}$ since $r\le 1$, and $c_2$ is some fixed constant).
\end{proof}

  \begin{proof}[Proof of Theorem \ref{p.supcrit}]
    We apply Theorem \ref{th:holdcont} to the SI-martingale $(\mu_n)_n$, $\Gamma=\Heisen/Z$ (endowed with the quotient metric).
    
    In the notations of Theorem \ref{th:holdcont}, $\kappa=2$. The only non-trivial hypothesis is \eqref{H:hyp4} and we will verify this using Lemma \ref{l.h_cont}. Let $u,v\in\Heisen/Z$. Identifying $\Heisen/Z$ with $\C$ and $u,v$ with $(u,0), (v,0)$, we consider $u,v$ also as elements of $\Heisen$ if necessary. Moreover, we use the notation $\eta_y$ for $\eta_{\pi(y)}$, for any $y\in\Heisen$.
    If we can show that
    \begin{equation}\label{eq:s.diff}
    \eta_v(B(x,r)\setminus v u^{-1} B(x,r))\le C|u-v|^\gamma\,,
    \end{equation}
for some constants $0<\gamma,C<\infty$, the hypothesis \eqref{H:hyp4} follows using the same argument as in \cite[Proposition 6.1]{ShmerkinSuomala}. Indeed, \eqref{eq:s.diff} implies the estimate (6.1) of \cite{ShmerkinSuomala} for the map $\Pi\colon Z\to\Heisen$, $\Pi_u(x)=ux$. See also Lemma \ref{l.technical}, where a similar estimate is derived in a more complicated situation.

Denote $x=(w,p)$. Since the map $y\mapsto x^{-1}y$ is a Heisenberg isometry and it maps vertical lines onto vertical lines, we have
\begin{align*}
\eta_v\left(B(x,r)\setminus v u^{-1} B(x,r)\right)&=\eta_{x^{-1}v}\left(B(0,r)\setminus x^{-1}vu^{-1} B(x,r)\right)\\
&=\eta_{ x^{-1}v}\left(B(0,r)\setminus z B(0,r)\right)\,,
\end{align*}
where $z=x^{-1}vu^{-1}x$. A simple calculation implies that $z=(a,b)$, where $|a|,|b|\le C|u-v|$. Thus, the map 
$y\mapsto zy$, $\pi^{-1}(u)\to\pi^{-1}(v)$ has the form
\[z(u,s)=(v,s+\varepsilon)\,,\]
where $\varepsilon\le C |u-v|$. It follows that
$\pi^{-1}(x^{-1}v)\cap z B(0,r)$ is a Euclidean translate of the line segment $\pi^{-1}(x^{-1}u)\cap B(0,r)$ tilted 
in the horizontal direction by (a Euclidean distance) at most $C|u-v|$. Since $\eta_{x^{-1}v}$ is the Lebesgue measure on the line 
$\pi^{-1}(\pi(x^{-1}v))$, it follows that
\begin{align*}
\eta_{x^{-1}v}\left(B(0,r)\setminus z B(0,r)\right)\le \varepsilon+|\eta_{x^{-1}v}\left(B(0,r)\right)-\eta_{x^{-1}u}\left(B(0,r)\right)|\le C|u-v|^{\tfrac12}\,,
\end{align*}
using Lemma \ref{l.h_cont} and the fact $\varepsilon\le|u-v|\le|u-v|^{\frac12}.$

Now that we have Theorem \ref{th:holdcont} at our disposal, we finish the proof by applying Lemma \ref{l.supcrit} or Lemma \ref{l.subcrit} according as $\beta > 2$ or $\beta \le 2$. If $\beta>2$ (that is, $\gamma<2$), Lemma \ref{l.supcrit} implies directly that $\pi\mu$ is absolutely continuous and that $\pi(E)$ has non-empty interior.
    
 If $\beta\le 2$ (i.e. $\gamma\ge 2$), we know that the assumption \eqref{eq:small_dim} of Lemma \ref{l.subcrit} holds for all $\theta >2-\beta$, but it still remains to verify \eqref{eq:uniform_tube_bound}. We use Lemma \ref{mu_m_to_mu_c.} as follows: Fix $\delta>0$ and denote $\theta=2-\beta+\delta$. Given $n\in\N$, consider a $2^{-n}$-dense family $\mathcal{D}_n\subset \pi(\Omega)\subset\Heisen/Z$ with cardinality at most $C 2^{2n}$. Lemma \ref{mu_m_to_mu_c.} applied to each $Z=\pi^{-1}(B(t,2^{-n}))$, $t\in\mathcal{D}_n$ (with $\varrho=\beta-\delta$) gives
 \begin{align*}
 &\PP\left(\pi\mu(B(t,2^{-n}))>4\left(\pi\mu_n(B(t,2^{-n}))+2^{n(\delta-\beta)}\right)\text{ for some }t\in\mathcal{D}_n\right)\\
 &\le C 2^{2n}\exp\left(-c2^{n\delta}\right)\,.
 \end{align*}  
Thus,
\[\sum_{n=1}^\infty \PP\left(\pi\mu(B(t,2^{-n}))>4\left(\pi\mu_n(B(t,2^{-n}))+2^{n(\delta-\beta)}\right)\text{ for some }t\in\mathcal{D}_n\right)<\infty\,.\] 
and by the Borel-Cantelli lemma, almost surely, there exists $N_0\in\N$ such that
\[\pi\mu(B(t,2^{-n}))\le 4\left(\mu_n(B(t,2^{-n}))+2^{n(\delta-\beta)}\right)\text{ for all }t\in\mathcal{D}_n, n\ge N_0\,.\] 
Replacing $4$ by (a random) constant $M<\infty$, the above remains true also for $1\le n<N_0$. Lemma \ref{l.subcrit} now implies that $\dim_H E\ge \dim(\pi(\mu))\ge\beta-\delta$ and letting
$\delta\downarrow0$ completes the proof.  Recall that $\pi$ is locally Lipschitz, so that it cannot increase the dimension of $E$ nor $\mu$. 
  \end{proof}

  \subsection{Dimension of random Heisenberg cutouts with respect to the Euclidean metric} \label{subs.final}
  Consider the continuous piecewise linear function $\phi : [0,4] \to
  [0,3]$
  \[ \phi(\beta) = \left\{ \begin{array}{ccc} x                       & \mathrm{if} & \beta \leq 2 \\
    2+\frac{1}{2}\left(\beta-2\right) & \mathrm{if} & \beta > 2    \\
  \end{array}
  \right. \text. \]
  It is a general fact (see \cite{Balogh2003}) that for any Borel subset
  $A \subset \mathbb{H}$, if we let $\beta$ (resp. $\alpha$) be the
  Hausdorff dimension of $A$ with respect to the Heisenberg
  (resp. Euclidean) metric, then
  \begin{equation}\label{eq.balogh} \alpha \leq \phi(\beta) \text.
  \end{equation}

  Our next Theorem states that for Heisenberg Poisson cut-outs, this is
  an equality with positive probability.

  We fix $\gamma \in ]0,4[$ and consider a random Heisenberg cut-out $E$ of parameter $\gamma$ as in the previous section. Let also $\mu$ be the cut-out measure and, as before, $\beta=4-\gamma$.
  Conditional on $\mu \neq 0$, $\beta$ is almost surely equal to the Hausdorff dimension of $E$ and $\mu$ (with respect to the Heisenberg metric).

  \begin{theorem}\label{th.main}
    With positive probability, the Hausdorff dimension of $E$ with respect to the Euclidean metric, $\dimh^\Eucl(E)$, is given by 
    \[ \dimh^\Eucl (E) = \left\{ \begin{array}{ccc}
      \beta & \mathrm{if} & \beta \leq 2 \\
      2 + \frac{1}{2}(\beta-2) & \mathrm{if} & \beta > 2\\
    \end{array} \right..
    \]
  \end{theorem}

  We first prove the following
  \begin{lemma}\label{lem:2moment}
    If $\beta > 2$ and $\mu \neq 0$, then with positive probability, $\dimh^\Heis (L_u\cap E)\geq \beta-2$ for (Lebesgue) positively many $u\in\R^2$.
  \end{lemma}
  Recall that $L_u$ is the vertical line $\pi^{-1}(u)$ for $u \in \Heisen/Z \simeq \R^2$.

  \begin{proof}
    Fix $u\in\R^2$. The restricted sequence $\mu_n|L_u$ is clearly an SI-martingale on $L_u$. Furthermore, denoting
    \[ \nu_n=\mu_n \eta_u =2^{(4-\beta)n} \eta_u|E_n\,,\]
    a standard calculation using the second moment method implies that for any $\varepsilon > 0$,
    \begin{align*}
      \EE\left(\limsup_{n\to\infty}\int\int d(x,y)^{2-\beta+\varepsilon}\,\mathrm{d}\nu_n\,\mathrm{d}\nu_n\right)<\infty\,,
    \end{align*}
    see e.g. \cite[Lemma 2.3]{Ojalaetal}. Thus, if $\nu$ is a weak*-limit of the sequence $\nu_n$, then almost surely,
    \[\int\int d(x,y)^{-t}\,\mathrm{d}\nu\,\mathrm{d}\nu<\infty\]
    for all $t<\beta-2$. Thus, $\dimh^\Heis(E\cap L_u)\ge \dim_H\nu\ge 2-\beta$ almost surely, provided $\nu\neq 0$.
    
    Note that the total mass of $\nu$ equals the random variable $X(u)$ from Theorem \ref{th:holdcont}. Since $\PP(X(u)>0)>0$ for all $u\in B^\circ(0,1)$, Fubini's theorem yields
    \begin{align*}
      \PP\times\mathcal{L}\left\{(E,u)\,:\,\dimh^\Heis(E\cap L_u)\ge 2-\beta\right\} & =\int\PP\left(\dimh^\Heis(E\cap L_u)\ge 2-\beta\right)\,\mathrm{d}u \\
      & \ge\int\PP(X(u)>0)\,\mathrm{d}u \\ & > 0\,.
    \end{align*}
    (where $\mathcal{L}$ is the Lebesgue measure on $\Heisen/Z \simeq \R^2$). 
    Thus 
    \[\mathcal{L}\{u\,:\,\dimh^\Heis(E\cap L_u)\ge 2-\beta\}>0\] 
    with positive probability.
  \end{proof}

  \begin{remark}
    Although we will not use it, we note that Lemma \ref{lem:2moment} actually holds in a much stronger form: Almost surely, $\dimh^\Heis(E\cap L_u)=\beta-2$ for \textbf{all} $u\in\R^2$ with $X(u)>0$, in particular for an open set of $u\in\R^2$, provided $\mu\neq 0$. This stronger form of dimension conservation may be derived using similar arguments as in \cite[Theorem 12.1]{ShmerkinSuomala}.
  \end{remark}

  \begin{proof}[Proof of Theorem \ref{th.main}]
    Assume first that $\beta\leq 2$. If $\mu \neq 0$, we know by Theorem \ref{p.supcrit} that almost surely, the projection $\pi(E)$ has Hausdorff dimension $\beta$;
     since the quotient mapping $\pi : \Heisen \to \Heisen/Z$ identifies with the projection mapping $\pi : \C \times \R \to \R$, and the latter is Lipschitz,
      we deduce that
    \[ \dimh^\Eucl(E) \geq \dimh(\pi(E))=\beta\]
    and the converse inequality is true because it is always true that $\dimh^\Eucl(E) \leq \dimh^\Heis(E)$, recall Lemma \ref{l.metcomp}, and $\dimh^\Heis(E)$ is almost
     surely $\leq \beta$ (Proposition \ref{pr.basic}).
    
    Now assume that $\beta>2$. Suppose that $\dim_H(E\cap L_u)\ge \beta-2$ for positively many $u\in\Heisen/Z$. By Lemma \ref{lem:2moment}, this is an event of positive probability.
    
    We are going to show that $\dimh^\Eucl(E) \ge 1+\tfrac{\beta}{2}$. We argue by contradiction and assume that for some $t<1+\tfrac{\beta}{2}$, the $t$-dimensional Hausdorff measure of $E$ with respect to the Euclidean metric,
    \[ \mathrm{H}^{t}_\Eucl(E) \]
    is finite.
    By Theorem 7.7 in \cite{Mattila}, we deduce that for Lebesgue-almost all $u\in\R^2$,
    \[ \mathrm{H}^{t-2}_\Eucl(E \cap L_u) < \infty \]
    (where $L_u$ is still the vertical line $\{u\} \times \R$). Hence, by Lemma \ref{l.metcomp} (\ref{i.vcomp}),
    \[ \mathrm{H}^{2t-4}_\Heis(E \cap L_u) < \infty \]
    for almost all $u$, which contradicts the fact that $\dimh^\Heis(E \cap L_u) = \beta-2>2t-4$ for positively many $u$  by Lemma \ref{lem:2moment}.
  \end{proof}

\begin{remark}
  Other families of fractals that enjoy equality in equation \eqref{eq.balogh} can be found in \cite{Balogh2003} (``horizontal fractals'') and \cite{Dufloux2016} (limit sets 
  of Schottky groups in ``good position'' at the boundary of the complex hyperbolic plane).

  We note that our result also provides an alternative to Theorem 1.7 in \cite{Balogh2003}: for any $\beta \in ]0,4[$, we construct a ``natural'' example 
  of a bounded Borel subset $A$ of $\Heisen$ with Hausdorff dimension $\beta$ with respect to the Heisenberg, such that the Euclidean Hausdorff dimension of $A$ 
  is equal to $\phi(\beta)$. 
\end{remark}

\section{Projections of Poisson cut-outs in $\SB^3$}\label{s.grom}
This is the main section of the paper. In the first subsection, \ref{ss.three-sph} we introduce the Euclidean metric $d_E$ on $\pdc$ and $\SB^3$ as well as the visual metric 
$d$ on $\SB^3$. In \ref{ss.chains} we define chains and state some useful Lemmas. In \ref{ss.projch} we look at the radial projection
along chains passing through a given point of $\SB^3$. In \ref{ineq} we show that the metric $d_E$ on the space of chains passing through a given 
point $x$ is comparable, away from $x$, to the Hausdorff distance between these chains seen as subsets of $\SB^3$. In \ref{ss.mainres} 
and \ref{ss.mainproof} we state and prove our main result (Theorem B from the Introduction). In \ref{ss.heis.sph} we recall why $\SB^3$, endowed with the visual distance, 
can be seen as the compactification of $\Heisen$, and how the disintegration, along chains passing through a point $x$, of the Lebesgue measure 
on $\SB^3$, is comparable to the $2$-dimensional Hausdorff measure on these chains. The last paragraph \ref{ss.tech} is devoted to a H\"{o}lder 
estimate for the measure of chains intersected with a small annulus.

\subsection{The $3$-sphere and its metrics}\label{ss.three-sph}
We endow $\C^3$ with two non-degenerate Hermitian forms: for $u=(u_0,u_1,u_2)$ and $v = (v_0,v_1,v_2)$ in $\C^3$, let
\[ u \cdot v = \sum_{i=0}^2 u_i \overline{v_i} \qquad \mathrm{and} \qquad \langle u,v \rangle = u_0 \overline{v_0} - u_1 \overline{v_1} - u_2 \overline{v_2} \text. \]
Note that $u \cdot v$ is the usual inner Hermitian product of $u$ and $v$. The Euclidean norm of $u$ is $\|u\| = \sqrt{u \cdot u}$. 
We denote by $q$ the quadratic form associated to $\langle \cdot,\cdot \rangle$, \emph{i.e.} $q(u)=\langle u,u \rangle$. 
The group of unimodular $q$-isometries $\mathbf{SU}(1,2)$ is denoted by $G$.

For any non-zero $w \in \C^3$, we denote by $w^\perp$ the set of all $u \in \C^3$ such that $\langle u,w \rangle=0$.

It is a general fact (see \cite{BourbakiA9}, \S 1.9) that if $E$ is a finite-dimensional complex vector space endowed with a non-degenerate Hermitian form $\Phi$, 
then for any $k \geq 1$ there is a canonical extension of $\Phi$ to the exterior product $\bigwedge^k E$, denoted by $\bigwedge^k \Phi$, such that 
\[ \bigwedge^k \Phi(u_1 \wedge \cdots \wedge u_k, v_1 \wedge \cdots \wedge v_k) = \mathrm{det} (\Phi(u_i,v_j)) \]
where the right-hand side is the usual determinant of the $k\times k$ matrix whose $(i,j)$-coefficient is $\Phi(u_i,v_j)$.

In this paper we will mostly consider the extension of the inner Hermitian product to $\bigwedge^2 \C^3$, defined by the relation 
\[ (u \wedge v) \cdot (u' \wedge v') = \left| \begin{array}{cc} u \cdot u' & u \cdot v' \\ v \cdot u' & v \cdot v'
\end{array}
\right| \]
and the corresponding Euclidean norm will be denoted by $\| u \wedge v\|$; it is characterized by the fact that $\| u \wedge v \| = \|u \| \cdot \| v\|$ if and only if $u \cdot v = 0$.

Let $\pdc$ be the complex projective plane,
\[ \pdc =  \{ [x_0:x_1:x_2]\ ;\ (x_0,x_1,x_2) \in \C^3 \setminus \{0 \} \} \]
where $[x_0:x_1:x_2]$ are the usual homogenous coordinates of $(x_0,x_1,x_2)$, so that $[x_0:x_1:x_2] = [\lambda x_0 : \lambda x_1 : \lambda x_2]$ for any $\lambda \neq 0$.
We will often use the same notation for elements of $\pdc$ and arbitrary lifts in $\C^3 \setminus \{0 \}$, 
and the letters $u,v,w,x,y,z$ may denote at the same time a non-zero vector in $\C^3$ or the corresponding point of $\pdc$. 
Likewise, we will usually denote by $w^\perp$ the complex projective line that is the image of $w^\perp \setminus \{0\}$ in $\pdc$.

We endow $\pdc$ with the metric defined by
\begin{equation}
   d_E(u,v) = \frac{\| u \wedge v \| }{\|u\|\cdot \|v\|} 
\end{equation}
Let us recall why $d_E$ satisfies the triangle inequality. When $v$ lies in the plane spanned by $u$ and $w$ this is easy to check; otherwise, 
let $v'$ be the orthogonal projection of $v$ onto that plane (with respect to the Hermitian inner product); it follows from the definition of $d_E$
(and the obvious inequality $\| v' \| \leq \|v\|$) that 
\[ d_E(u,v') \leq d_E(u,v) \]
and likewise with $w$ instead of $u$; we thus have 
\[ d_E(u,w) \leq d_E(u,v') + d_E(v',w) \leq d_E(u,v) + d_E(v,w) \]
which is the triangle inequality.

The definition of $d_E$ above is equivalent to 
\begin{equation}
  d_E(u,v)^2 = 1 - \frac{|u \cdot v|^2}{\|u\|^2 \cdot \|v\|^2} 
\end{equation}
which shows that $d_E$ is the sinus of the angle metric. It is therefore biLipschitz equivalent to the usual Riemannian metric on $\pdc$.

In $\pdc$ we consider the $3$-sphere
\[ \SB^3 = \{ [1 : x_1 : x_2 ] \in \pdc\ ;\ |x_1|^2+|x_2|^2=1 \} = \{ u \in \pdc\ ;\ q(u)=0 \} \text. \]

On $\SB^3$ the restriction of $d_{\mathrm{E}}$ is biLipschitz-equivalent to the usual Euclidean metric, but we are more interested in 
the visual metric $d$ which we now define:
\begin{equation}\label{eq.grmet} d(u,v) = \sqrt{\frac{|\langle u,v \rangle|}{\|u \| \cdot \| v \|}} \end{equation}
for any $u,v \in \SB^3$. If $\SB^3$ is viewed as the boundary of the complex hyperbolic plane $\mathbb{H}^2_\C$, then $d$ is the visual metric 
associated to the hyperbolic metric on $\mathbb{H}^2_\C$. See \emph{e.g.} \cite{LNQuint}. Note that if $\SB^3$ was identified with the boundary 
of the \emph{real} hyperbolic $4$-space, the corresponding visual metric would be biLipschitz-equivalent to $d_E$.

Balls of $\SB^3$ with respect to $d$ will be denoted by $B(x,r)$ for $x \in \SB^3$ and $r>0$. Such a ball will sometime be called a ``visual ball''.

Let $\mathcal H$ be the usual Lebesgue measure on $\SB^3$. One may check (Lemma \ref{l.reg-sph})
that if $f(r)$ is
the measure of a visual ball of radius $r$,
then
\[\lim_{r\downarrow 0}\frac{f(r)}{r^4}=a\,,\]
for some $0<a<\infty$. 

For convenience, we normalize $\mathcal{H}$ such that $a=1$. Note that the measure of a Euclidean ball of radius $r$ is equal, 
up to some multiplicative constant, to $r^3$ for any $r$ small enough.

In the next Lemma we state some easy facts which we will use freely. Recall that $G = \mathbf{SU}(1,2)$ acts on $\SB^3$ (because it preserves $q$). For $g\in G$, we benote by $||g||$ the usual operator norm of $g$.
\begin{lemma}\label{lemma.gromov}
  In $\SB^3$ endowed with the visual metric $d$:
  \begin{enumerate}
    \item\label{i.grom1} For any $u,v \in \SB^3$,
    \[ d(u,v)^2 \leq d_{\mathrm{E}}(u,v) \lesssim d(u,v) \text. \]
    \item\label{i.grom2} For any $u,v \in \SB^3$ and $g \in G$,
    \[ d(gu,gv) = d(u,v) \sqrt{\frac{\|u\|}{\|gu\|} \frac{\|v\|}{\|gv\|}} \]
    and 
    \begin{equation}\label{eq.gromlip}
      \frac{d(gu,gv)}{d(u,v)} \leq 1 + \| \Id-g^{-1} \|
    \end{equation}
    \item\label{l.grom3} For any $x \in \SB^3$, and $g \in G$,
    \begin{equation}
      d(x,gx) \leq \sqrt{ \|\Id - g \| \cdot \|g^{-1} \|}
    \end{equation}
    \item\label{l.grom4} For $g \in G$, $x \in \SB^3$ and $r > 0$,
    \begin{equation}\label{eq.balldist}
      gB(x,r) \subset B\left(x,r+\sqrt{\|\Id-g\| \cdot \|g^{-1} \|}\right)\,.
    \end{equation}

  \end{enumerate}
\end{lemma}

\begin{proof}

Statement \eqref{i.grom1} is verified by an elementary computation that we omit. First part of statement \eqref{i.grom2} follows from the definition of $d$, see \eqref{eq.grmet}, and implies the second part.
Let us prove statement \eqref{l.grom3} briefly; because $x$ belongs to $\SB^3$, we have $\langle x,x \rangle = 0$ so $|\langle x,gx \rangle| = |\langle x,gx-x \rangle | \leq \| x\| \| x - gx \|$.
Hence 
\[ d(x,gx) \leq \sqrt{ \frac{\| x \|^2 \|\Id-g \|}{\|x\| \|gx \|}} \leq \sqrt{\| Id - g \| \cdot \| g^{-1} \|} \]
where we used the fact that $\|x\| \leq \|g^{-1} \| \cdot \|gx \|$. Finally, \eqref{l.grom4} follows from \eqref{l.grom3} and the triangle inequality.
\end{proof}

\subsection{Chains}\label{ss.chains}
  \begin{defn}
    If $L \subset \pdc$ is a (complex) projective line which meets $\SB^3$ in more than one point, we say that the intersection
    \[ L \cap \SB^3 \]
    is a \emph{chain}.
  \end{defn}

  If follows from the definition that \emph{through two points of $\SB^3$ there passes one and only one chain}. Chains are not geodesics, 
  though, and the reader should not think that they minimize length in any way. 

  If $x$ is a fixed point of $\SB^3$, the family of all chains passing through $x$ yields a foliation of $\SB^3 \setminus \{ x \}$, the leaves 
  of which are of the form $L \setminus \{ x \}$, where $L$ is a chain passing through $x$. We will shortly (see \ref{ss.projch}) provide an explicit family 
  of projections $(\pi_x)_{x \in \SB^3}$ such that the fibres of $\pi_x$ are the chains passing through $x$ (with $x$ removed) 
  and $\pi_x$, restricted to any  compact subset of $\SB^3 \setminus \{ x \}$, is a Lipschitz mapping into $\SB^2$ (with one point removed).

\begin{remark}
At this point, it is perhaps useful to draw the reader's attention to the fact that in the Euclidean sphere $\SB^3$ with a fixed point $x$, the family of all 
small circles (in the usual sense) passing through $x$ does \emph{not} yield a foliation of $\SB^3 \setminus \{ x \}$, because a single point $y$ belongs to several (indeed, infinitely many)
small circles passing through $x$. In order to obtain a foliation, one would have to fix both a point $x$ and a direction in the $3$-dimensional space 
tangent to $\SB^3$ at $x$.

The chains we are considering are very special ``small circles'': they are the small circles which are the boundaries of 
totally geodesic submanifolds of the complex hyperbolic space $\mathbf{H}^2_\C$ of curvature $-4$ (assuming the complex hyperbolic metric is normalized to have 
curvature between $-4$ and $-1$).

A crucial property of chains is that they are Ahlfors-regular of dimension $2$ with respect to the restriction of the visual metric. See Lemma \ref{eucl-vis-comp} and the discussion thereafter. 
\end{remark}

It is easy to check  
  that any chain is of the form $w^\perp \cap \SB^3$, where $w \in \C^3 \setminus \{0 \}$ is such that 
  $q(w)<0$, and the projective class of $w$ is uniquely defined. 
  We will denote by $\Chains$ the space of all chains; $\Chains$ identifies with the space of all $w \in \pdc$
  such that $q(w)<0$ (where we denote by $w$ both an element of $\pdc$ and some lift of this element in $\C^3$):
  \[ \Chains = \{ w\in \pdc\ ;\ q(w)<0 \} \text. \]
This space will be endowed with the restriction of $d_{\mathrm{E}}$. We will use the letter $L$ to denote a chain. The chain $w^\perp \cap \SB^3$, where 
$w \in \Chains$, will be denoted by $L_w$.

\begin{lemma}\label{lemcomp}
  Let $\mathcal K$ be a compact set of chains, \emph{i.e.} a compact subset of $\Chains$. Then for any fixed chain $w_0 \in \Chains$,
   there is a compact subset $K$ of $G$
  such that $\mathcal K = K \cdot w_0$.
\end{lemma}
\begin{proof}
  The operation of $G$ on $\Chains$ is transitive (because of Witt's transitivity Theorem, see \emph{e.g.} \cite[\S 4.3]{BourbakiA9}) and smooth.
  Let $H$ be the stabilizer of $w_0$ in $G$; then the quotient $G/H$ is homeomorphic to $\Chains$ (because $G$ is a Lie group); 
  and  any compact subset of this quotient space
  can be lifted to a compact subset of $G$. Indeed let $y$ be some point of $G/H$ and fix a lift $x$ of $y$ in $G$;
  there is a compact neighbourhood $V$ of $x$ in $G$ and the image of $V$ in $G/H$ is a neighbourhood of $y$ (because the mapping $G \to G/H$ is open) that is also compact.
  Hence we have found a compact lift of a small compact neighbourhood of $y$. Using the local compactness of $G/H$ we see that any of its compact 
  subsets can indeed be lifted to a compact subset of $G$. Hence the Lemma.
\end{proof}

The previous Lemma will allow us to prove metric estimate for compact sets of chains by considering a fixed chain, carrying out explicit computations,
and then using Lemma \ref{lemma.gromov} to move the chain around using elements $g$ in a compact subset of $G$, only losing some bounded multiplicative 
constants in the process.

We will also need the following more precise version:

\begin{lemma}\label{subl1}
  Fix a chain $L_0 \in \chsp$. There is a neighbourhood $\mathcal U$ of $L_0$ in $\chsp$ such that for any $L \in \mathcal U$
  there is $g \in G$ mapping $L_0$ onto $L$ and satisfying $\| \Id - g \| \asymp d_E(L_0,L) \asymp \| \Id - g^{-1} \|$.
\end{lemma}
\begin{proof}
  As before, let $H$ be the stabilizer of $L_0$ in $G$; the homeomorphism  $\phi: G/H \to \chsp$, considered in the proof of the previous Lemma,
   is locally biLipschitz when $G/H$ is endowed with the usual Riemannian metric, that is the 
  quotient, by $H$, of the right-invariant Lie group metric on $G$.

  Fix a neigbourhood $\mathcal U$ of $L_0$ where the restriction of $\phi^{-1}$ is biLipschitz,
  and small enough that there is a smooth section $\sigma : \phi^{-1} (\mathcal U) \to G$ that maps the image of $\Id$ in $G/H$ to $\Id$.
    Recall that smooth sections exist locally because $G$ is a Lie group and $H$ is closed.
  
  Then $\sigma \circ \phi^{-1}$ is a biLipschitz mapping from $\mathcal U$ onto its image. If $L$ belongs to $\mathcal U$ and $g = \sigma \circ \phi^{-1}(L)$,
  then $L=gL_0$ by definition of $\sigma$ and $\phi$, and $\| \Id-g \| \asymp d(L,L_0)$ because the operator norm is locally Lipschitz equivalent to
  the right-invariant Lie group metric of $G$.
\end{proof}

\subsection{Radial projection along chains}\label{ss.projch}
  To any $x \in \SB^3$ we are going to associate a projection mapping from $\SB^3 \setminus \{ x \}$ 
  into the Euclidean $2$-sphere:
  \[ \pi_x : \SB^3\setminus\{x\} \to \SB^2 \]

  If we call $x$ the ``direction'' of the projection, we can then study Hausdorff dimension of projections in some direction, 
  in almost every direction, or in every direction.

  Let $x \in \SB^3$. The orthogonal $x^\perp$ (in $\pdc$) is a complex projective line tangent to $\SB^3$ at $x$. For any $y \in \SB^3$ distinct 
  from $x$, the projective lines $y^\perp$ and $x^\perp$ have a single intersection point, which belongs to $\Chains$ (\emph{i.e.} if $u$ is a lift 
  of this element to $\C^3$, then $q(u)<0$). 

  Let $\pi_x(y)$ be this intersection point. Then $\pi_x(y)^\perp \cap \SB^3$ is the chain passing through $x$ and $y$. 
  Although $x$ and $y$ play symmetric roles (so that $\pi_x(y)$ is actually equal to $\pi_y(x)$), our notation emphasizes the fact that we see $\pi_x(y)$ as an element of the projective line
  $x^\perp$. This is, by definition, the projection of $y$ in the ``direction'' of $x$. Note that
  \[ \pi_x : \SB^3 \setminus \{x\} \to x^\perp \setminus \{x\} \]
  is onto and \emph{the fibres of this mapping  are the chains passing through $x$}; this is, of course the main point here:
  the geometric structure we are interested in is the family of foliations of $\SB^3$ by chains, and projection mappings are but a tool to study the geometry
  of our cut-out sets with respect to these foliations.

  The codomain $x^\perp \setminus \{x\}$  is endowed with the restriction 
  of $d_{\mathrm{E}}$. Note that $x^\perp$ is a complex projective line, so that we have indeed defined a mapping from $\SB^3 \setminus \{x\}$
  onto a Euclidean $2$-sphere with one point removed.
\begin{lemma}
    The restriction of $\pi_x$ to any compact subset of $\SB^3 \setminus \{x\}$ is Lipschitz when  $\SB^3$ is endowed with the restriction of either $d$ or $d_{\mathrm{E}}$.
\end{lemma}

\begin{proof}
	For $d_{\mathrm{E}}$, this is proved in \cite[Proposition 1]{Dufloux2017}. Recalling Lemma \ref{lemma.gromov} \eqref{i.grom1}, this holds also for the visual metric $d$. 
\end{proof}

We will need the following 
  \begin{lemma}\label{eucl-vis-comp}
    Let $x,y$ be in $\SB^3$ and consider $\pi_x(y)$ as an element of $\C^3$. Then 
    \[ \frac{\langle \pi_x(y) , \pi_x(y) \rangle}{\| \pi_x(y) \|^2} = -   \frac{|\langle x,y\rangle|^2}{\| x \wedge y \|^2} \]
  \end{lemma}
Note that the left-hand side is well-defined and equal to $\frac{\langle w,w\rangle}{\|w\|^2}$ for any representative $w$ of $\pi_x(y)$ in $\C^3$.

In this Lemma, the left-hand side depends only on the complex projective line passing through $x$ and $y$; in particular, if $L$ is a fixed chain, 
then for any distinct $x,y \in L$ the number 
\[ \frac{|\langle x,y\rangle|^2}{\| x \wedge y \|^2} \]
depends only on $L$. This is a quantitative version of the fact that along chains, $d_E^2$ is comparable to $d$.

\begin{proof}
  Consider the mapping $\kappa : \bigwedge^2 \C^3 \to \C^3$ defined by the relation 
  \[ \langle \kappa(u \wedge v), w \rangle e_0 \wedge e_1 \wedge e_2 = u \wedge v \wedge w \]
  for any $u,v,w \in \C^3$, where $(e_0,e_1,e_2)$ is the canonical basis of $\C^3$. It is easy to check that $\kappa$ is an isometry when $\C^3$ is
  endowed with either $\langle \cdot,\cdot \rangle$ or the Hermitian inner product, and $\bigwedge^2 \C^3$ is endowed with the corresponding extension.

  Also, by definition $\kappa(u \wedge v)$ is a representative, in $\C^3$, of $\pi_x(y)$ if $u$ (resp. $v$) is a representative of $x$ (resp. $y$).
  We thus have 
  \[ \frac{\langle \pi_x(y),\pi_x(y) \rangle}{\|x \wedge y\|^2} = \frac{\langle \kappa(x \wedge y),\kappa(x \wedge y) \rangle}{\| \kappa(x \wedge y) \|^2} \]
  and this proves the Lemma because using the relations $q(x)=q(y)=0$ one sees that $\langle x\wedge y, x \wedge y \rangle = -|\langle x,y \rangle |^2$.
\end{proof}

\subsection{Metric inequalities}\label{ineq}
\subsubsection{Distance to a chain}

\begin{lemma}\label{haus-two-frostman}
  Let $\mathcal K$ be a compact set of chains. For any $L \in \mathcal K$, let $\eta_L$ be the $2$-dimensional Hausdorff measure with 
  respect to the restriction $d|L$. Then uniformly in $L \in \mathcal K$, 
  \[ \eta_L(B(x,r)) \lesssim r^2 \]
  for any $x \in \SB^3$ and $r>0$.
\end{lemma}
\begin{proof}
  If $L$ satisfies the conclusion of the Lemma (for any $x$ and $r$), and $g$ belongs to some compact subset of $G$, then $gL$ also
  satisfies the conclusion of the Lemma, with a new constant that depends continuously on $g$. Indeed, $g$ yields a biLipschitz mapping 
  from $L$ to $gL$,see Lemma \ref{lemma.gromov} \eqref{i.grom2}; the $2$-dimensional Hausdorff measure on $gL$ is thus equivalent 
  to the push-forward, through $g$, of the $2$-dimensional Hausdorff measure on $L$, and the Radon-Nikodym density is bounded. 

  Now if $L$ is the chain orthogonal to, \emph{e.g.}, $w=[0:0:1]$, an easy computation shows that indeed $\eta_L(B(x,r)) \lesssim r^2$ for any 
  $x \in \SB^3$ and $r>0$. Hence the Lemma follows from an application of Lemma \ref{lemcomp}.
\end{proof}

\begin{lemma}\label{l.comp-trans}
  Fix a compact subset $\mathcal K$ of $\chsp$. For $w \in \mathcal K$ such that $q(w) < 0$ we denote by 
  $L_w$ the corresponding chain, \emph{i.e.} $L_w=w^\perp \cap \SB^3$.
  
  The following holds uniformly in $x\in \SB^3$ and $w \in \mathcal K$:
  \begin{equation}
    d(x,L_w) \asymp d_E(x,L_w) \asymp d_E(x,w^\perp) = \frac{|\langle x,w \rangle|}{\|x\|\cdot \|w\|}
  \end{equation}
\end{lemma}
Here, $d(x,L_w)$ is the visual distance from $x$ to $L_w$, that is, 
\[ d(x,L_w) = \inf_{y \in L_w} d(x,y) \]
Likewise $d_E(x,L_w)$ is the corresponding Euclidean distance (the same quantity, where $d_E(x,y)$ is replaced by $d(x,y)$) and $d_E(x,w^\perp)$ is the Euclidean distance from $x$ to $w^\perp$ in $\pdc$.

The content of this lemma is two-fold: first, $d$ and $d_E$ are comparable transversally to chains (compare to Lemma \ref{l.metcomp} \eqref{trans_comb}); second, this transversal distance is 
given by the simple formula above.

\begin{proof}
Fix $w \in \C^3$ such that $q(w) < 0$. For any $x \in \pdc$, the formula 
\[ d_E(x,w^\perp) = \frac{|\langle x,w \rangle|}{\|x\| \cdot \|w\|} \]
is well-known and easy to check. Also, the inequality $d_E(x,L_w) \geq d_E(x,w^\perp)$ is obvious.

Now fix $w_0 = [0:0:1]$ and let $x$ be some element of $\SB^3$. A simple calculation shows that $d_E(x,L_{w_0}) \lesssim d_E(x,w_0^\perp)$.
If $g$ is some element of $\mathbf{SU}(1,2)$ and we denote by $w$ the image $gw_0$, then we have, for any $x \in \SB^3$,
\[ d_E(x,L_w) \lesssim d_E(g^{-1}x, L_{w_0}) \lesssim d_E(g^{-1}x, w_0^\perp) \lesssim d_E(x, gw_0^\perp) = d_E(x,w^\perp) \]
where the constants depend on the operator norms of $g$ and $g^{-1}$.

If $\mathcal K$ is a fixed compact subset of $\Chains$, there is a compact subset $K$ of $G$ such that $\mathcal K = K \cdot w_0$. The previous 
argument then gives 
\[ d_E(x,L_w) \lesssim_{\mathcal K} d_E(x,w^\perp) \]
for any $x \in \SB^3$ and $w \in \mathcal K$.

Similarly, one can check that for any $x \in \SB^3$, $d(x,L_{w_0}) \lesssim d_E(x,L_{w_0})$ and then for $g \in G$, 
\[ d(x,L_{gw_0}) \lesssim d(g^{-1}x,L_{w_0}) \lesssim d_E(g^{-1}x,L_{w_0}) \lesssim d_E(x,L_w) \] 
where the constants depend on the operator norms of $g$ and $g^{-1}$, and we argue as before.
\end{proof}

\begin{lemma}\label{lem_a}
  Let $K_1,K_2$ be non-empty disjoint compact subsets of $\SB^3$. For any $x \in K_1$, $y \in K_2$ and any chain $L_w$ ($w \in \Chains$)
  passing through $x$ and $K_2$, 
  \[ d_E(\pi_x(y), w) \asymp d(y,L_w) \]
\end{lemma}
\begin{proof}
Let $K\subset\SB^3$ be a compact set such that $K\cap K_1=\varnothing$ and the $\delta$-neighbourhood $K_2(\delta)=\cup_{u\in K_2}B(u,\delta)\subset K$ for some $\delta>0$.  Recall that the restriction of $\pi_x$ to $K$ is Lipschitz, with a uniform Lipschitz constant when $x \in K_1$, when $K$ is endowed with the
restriction of (the Euclidean or) the visual metric. This yields at once the inequality 
\begin{equation*}\label{yksviela} d_E(\pi_x(y), w) \lesssim d(y,L_w \cap K)\,. 
\end{equation*}
Because $K_1$ and $K_2$ are disjoint (and $y$ belongs to $K_2$), the right-hand side is comparable to $d(y,L_w)$. Hence, we obtain $d_E(\pi_x(y), w) \lesssim d(y,L_w)$ and 
what is left is to prove the converse inequality.

Recalling Lemma \ref{l.comp-trans}, it suffices to prove that 
\begin{equation}\label{eq.lemdisj} \frac{|\langle y,w \rangle|}{\| y \| \| w \|} \lesssim \frac{\| w \wedge \pi_x(y) \|}{\|w\| \|\pi_x(y) \|}\,. \end{equation}

First, remark that the exterior product $\bigwedge^3 \C^3$ is a complex line; 
it is readily checked that the canonical extension of $\langle \cdot , \cdot \rangle$ to this complex line is equal to the extension of the Hermitian inner product. 
We thus have 
\begin{align*} 
\| w \wedge \pi_x(y) \wedge y \|^2 &= \left\langle w \wedge \pi_x(y) \wedge y , w \wedge \pi_x(y) \wedge y \right\rangle\\ &= \left| 
\begin{array}{ccc}
  \langle w , w \rangle & \langle w , \pi_x(y) \rangle & \langle w , y \rangle \\
  \langle \pi_x(y) , w \rangle & \langle \pi_x(y) , \pi_x(y) \rangle & \langle \pi_x(y) , y \rangle \\
\langle y , w \rangle &  \langle y , \pi_x(y) \rangle & \langle y , y \rangle \\ \end{array} \right|\,.
\end{align*}

By definition, $y$ is orthogonal to $\pi_x(y)$ and to $y$ itself, from this, it follows at once that the above determinant is equal to 
\[ -| \langle w ,y \rangle|^2 \times  \langle \pi_x(y) , \pi_x(y) \rangle\]
Now use the fact (Lemma \ref{eucl-vis-comp}) that 
\[ \langle \pi_x(y) , \pi_x(y) \rangle= - \| \pi_x(y) \|^2  \times \frac{| \langle x,y \rangle|^2}{\|x \wedge y \|^2}\,. \]
All in all, we thus have 
\[ \| w \wedge \pi_x(y) \wedge y \| = \| \pi_x(y) \| \times | \langle w , y \rangle | \times \frac{|\langle x ,y \rangle|}{\| x \wedge y \|}\,. \]

We can now prove inequality \eqref{eq.lemdisj}. The above computations yields 
\[ \frac{\| w \wedge \pi_x(y) \wedge y \|}{\|w\| \| \pi_x(y) \| \| y \|} = \frac{| \langle w , y \rangle |} {\|w\| \|y \|} \times \frac{d(x,y)^2}{d_E(x,y)} \]
and for $x \in K_1$, $y \in K_2$ the distance $d(x,y)$ is uniformly bounded below by some positive constant, while $d_E(x,y)$ is bounded above by
 $1$.
We thus have 
\[ \frac{\| w \wedge \pi_x(y) \wedge y \|}{\|w\| \| \pi_x(y) \| \| y \|} \gtrsim \frac{| \langle w , y \rangle |}{\|w\| \|y \|} \]
and the left-hand side is bounded above by $d_E(w,\pi_x(y))$ because, for any $u_1,u_2,u_3 \in \C^3$, one has 
\[ \| u_1 \wedge u_2 \wedge u_3 \| \leq \| u_1 \| \times \| u_2 \wedge u_3 \| \]
and this finishes the proof.
\end{proof}

\begin{cor}\label{cor_a}
  Let $K_1,K_2$ be non-empty disjoint compact subsets of $\SB^3$ and let $K$ be a compact set such that $K\cap K_1=\varnothing$ and $K_2(\delta)\subset K$. For any $x \in K_1$ and any $u,v \in \Chains$ such that $L_u$ and $L_v$ both 
  pass through $x$ and $K_2$,
  \[ d_E(u,v) \asymp d(L_u \cap K, L_v \cap K) \]
  where the right-hand side denotes the Hausdorff distance between the $L_u \cap K$ and $L_v \cap K$, with respect to either the visual 
  or the Euclidean metric. 
\end{cor}
Recall that if $A,B$ are closed subsets of some metric space $X$, the Hausdorff distance $d(A,B)$ is the number 
\[ \max\{\theta(A,B), \theta(B,A)\}\]
where 
\[ \theta(A,B) = \sup_{x \in A} d(x,B) \]
\begin{proof}
  The previous Lemma (recall \eqref{yksviela}) shows that if $u,v$ are as in the statement of the Corollary, then for any $y \in L_u \cap K_2$, 
  \[ d(y,L_v \cap K) \asymp d_E(u,v) \]
  and in particular the supremum, for $y \in L_u \cap K$, of the left-hand side, is comparable to $d_E(u,v)$. With the notation above, we thus have 
  \[ \theta(L_u \cap K, L_v \cap K) \asymp d_E(u,v) \]
  and the corollary follows from this.
\end{proof}

The content of the above results should be clear: it is a generalization of the fact that in Heisenberg group, the Hausdorff distance between two 
vertical chains is the same when computed with respect to either the Euclidean or the Heisenberg metric, and it is also equal (by definition) 
to the distance between the images of these vertical chains in the quotient space $\Heisen/Z$. We are now replacing the vertical projection with the 
radial projection with respect to any point, simply losing some multiplicative constants in the process.

The following lemma will be needed in the course of the proof of Lemma \ref{l.main-lemma}.

\begin{lemma}\label{technical-lemma}
 Let $B,V$ be non-empty disjoint compact subsets of $\SB^3$. There is a constant $C > 0$ such that the following holds: for any $x,x' \in B$, 
any chains $L_u,L_{u'}$ passing through $x,x'$ respectively and also meeting $V$, and for any $r > 0$, 
\[ \pi_x^{-1}(B(u,r)) \cap V \subset \pi_{x'}^{-1}(B(u',C(r+d_E(u,u'))))\,. \] 
\end{lemma}
\begin{proof}
Note that $u\in x^\perp$ and $u'\in x'^\perp$. The claim follows from the previous results along with the triangle inequality for the Hausdorff metric. 
Let $V'$ be a compact set disjoint form $B$ such that it contains the $\delta$-neighbourhood $V(\delta)$ for some $\delta>0$.
Then, we know from Corollary \ref{cor_a} that $d(L_u \cap V', L_{u'} \cap V')$ 
is comparable to $d_E(u,u')$. Thus, for all $y\in\pi_x^{-1}(B(u,r)) \cap V$, the triangle inequality along with Lemma \ref{lem_a}, \eqref{yksviela}, and Corollary \ref{cor_a} yields
\begin{align*} 
d_E(\pi_{x'}(y),u')&\asymp d(y,L_{u'} \cap V') \leq d(y,L_u \cap V')+ d(L_u \cap V', L_{u'} \cap V') \\
&\asymp d_E(\pi_x(y),u) + d_E(u,u')\le r+d_E(u,u') 
\end{align*}
and this is equivalent to the required inclusion.
\end{proof}

\subsection{Statement of the main result}\label{ss.mainres}

Just like in the previous section, we can define random Poisson cut-outs in $\SB^3$ \emph{with respect to the visual metric}. 
We obtain a random cut-out set $E$ and a random finite Borel measure $\mu$, supported on $E$, and non-zero with positive probability.

Let $\gamma$ be the intensity parameter of the cut-out as in Section \ref{sec2}. Conditional on $\mu \neq 0$, we know that, almost surely,
$\dimh(E)=4-\gamma$ and, also, $\mu$ has exact dimension $4-\gamma$.

Our main Theorem deals with the behaviour of the cut-out set with respect to radial projections along chains in every ``direction'',
\emph{i.e.} along $\pi_x$ for \emph{every} $x \in \SB^3$. This extends the corresponding results for projections of Euclidean cut-out sets
along \emph{every} orthogonal projection \cite{ShmerkinSuomala}.

\begin{theorem}\label{th.mainres}
  Let $E$ be a random Poisson cut-out set in $\SB^3$ and let $\mu$ be the cut-out measure. 

  Let $\beta$ be the Hausdorff dimension of $E$ (with respect to the visual metric).
  Then, almost surely on $\mu\neq0$, the following holds: For every $x \in \SB^3$, 
  \[ \dimh (\pi_x (E)) = \dim (\pi_x \mu) = \inf \{2, \beta\}\]
  and, if $\beta > 2$, $\pi_x(\mu)$ is absolutely continuous and $\pi_x(E)$ has non-empty interior. 
\end{theorem}

Theorem \ref{th.mainres} will be proved in Section \ref{ss.mainproof}.

\begin{remark}\label{rem:marstrand_fails}
As explained in the introduction to this paper, it is not true that if $A$ is a Borel subset of $\SB^3$ of Hausdorff dimension 
$\beta$ with respect to the visual metric, and if we pick $x$ at random with respect to the Lebesgue measure on $\Heisen$, then 
the image $\pi_x(A)$ has almost surely Hausdorff dimension $\inf\{2,\beta\}$. 

  For instance, any chain $L\subset\SB^3$ has (visual) Hausdorff dimension $2$, but  all of its radial projections $\pi_x(L)$ are smooth curves
   (or singletons if $x\in L$) so their dimension is $\le1$.

However, in the special case when $\alpha$, the Hausdorff dimension of $A$ with respect    
to the Euclidean metric, is given by 
\[ \alpha = \phi(\beta) \]
(where $\phi$ is as in \eqref{subs.final}), it is true that $\pi_x(A)$ has Hausdorff dimension $\inf\{2,\beta\}$ for almost all $x$; this follows 
at once from Theorem 5 in \cite{Dufloux2017}. The Theorem \ref{th.mainres} shows that for Poisson cut-outs, we have a much stronger result: We can replace ``almost all $x\in\SB^3$'', by ``all $x\in\SB^3$''.   
\end{remark}

\subsection{Relating $\SB^3$ to Heisenberg group}\label{ss.heis.sph}
It is well-known that the Euclidean sphere $\SB^n$ minus one point $x$ is mapped onto the Euclidean space $\R^n$ through the so-called 
stereographic projection. This mapping is one-to-one and conformal. Small circles of $\SB^n$ passing through $x$ are mapped onto affine lines of $\R^n$.

Likewise, the visual sphere $\SB^3$ minus $x$ is mapped onto the Heisenberg group $\Heisen$; this mapping is locally biLipschitz, and chains passing 
through $x$ are mapped onto vertical lines, that is, translates of the center $Z = \R \times \{0\}$. We will now define this mapping and 
derive some useful results.

The operation of $\mathbf{SU}(1,2)$ on $\C^3$ passes to the quotient and gives an operation of $\mathbf{PU}(1,2)$ on $\pdc$. Since $\SB^3$ is the 
set of all $w \in \pdc$ such that $q(w)=0$, the operation of $\mathbf{PU}(1,2)$ on $\pdc$ can be restricted to  the invariant subset $\SB^3$. 

Now fix a point $x \in \SB^3$ and let $P_x$ be the stabilizer of $x$ in $\mathbf{PU}(1,2)$. The unipotent transformations in $P_x$ 
form a subgroup isomorphic to $\Heisen$. The operation of $\Heisen$ on $\SB^3 \setminus \{x \}$ is simply transitive, allowing for an identification 
of $\SB^3 \setminus \{x\}$ with $\Heisen$. We refer the reader to \cite[Chapter 4]{Goldman} for details and for explicit descriptions of $P_x$ and $\Heisen$ 
in appropriate coordinates (using an Iwasawa decomposition of $\mathbf{PU}(1,2)$). Another useful (and more accessible) reference is \cite{LNQuint}. See also \cite[pp. 47--55]{Capo}.

The identification of $\Heisen$ with $\SB^3 \setminus \{x\}$ depends on the choice of a point in $\SB^3 \setminus \{x\}$ (this is the point that will be identified 
with the origin of $\Heisen$). One way to choose this point is to let $o=[1:0:0]$ be the base point in the $4$-ball $\mathbf{B}^4=\{ w\in \pdc\  ;\ q(w)>0 \}$; 
the stabilizer $K$ of $o$ in $\mathbf{PU}(1,2)$ identifies with $\mathbf{U}(2)=\mathbf{SO}(3)$ and the stabilizer of $x$ in $K$ identifies with $\mathbf{SO}(2)$
and fixes exactly two points: $x$ and $\hat x \in \SB^3$. 
We let $\hat x$ be the point of $\SB^3$ associated to the origin of $\Heisen$. (This identification of $\Heisen$ with $\SB^3 \setminus \{x\}$ is uniquely defined up to conjugation by 
an element of $\mathbf{SO}(2)$, that is, up to a Euclidean rotation with axis $Z$.)

Let $\phi_x : \SB^3 \setminus \{x\} \to \Heisen$ be the mapping we just defined. This is a ``Heisenberg stereographic projection at $x$''.

\begin{prop}\label{p.stereo}
  \begin{enumerate}
    \item 	The Heisenberg stereographic projection $\phi_x$ maps chains passing through $x$ (with $x$ removed) onto vertical lines 
    in $\Heisen$. Any vertical line 
    in $\Heisen$ is the image of one and only one chain passing through $x$. 
    \item Fix $x \in \SB^3$ and let $K$ be a compact subset of $\SB^3 \setminus \{x\}$. There is a constant $C > 0$ such that 
    \begin{itemize}
      \item for any $y,y' \in K$, $ C^{-1} d(y,y') \leq d(\phi_x(y),\phi_x(y')) \leq C d(y,y')$ (where as before we use the symbol
       $d$ for both the visual metric on $\SB^3$ and the 
      \Kor metric on $\Heisen$).
      \item The push-forward of the Lebesgue measure on $K$ through $\phi_x$ is equivalent to the Lebesgue measure on $\phi_x(K)$, and the Radon-Nikodym
      derivative is continuous and lies between $C^{-1}$ and $C$.
    \end{itemize}
  \end{enumerate}
\end{prop}
\begin{proof}
  For the first point, see \cite[4.2.3.]{Goldman},  The second point follows from the fact that the push-forward of the visual metric through 
  $\phi_x$ is locally biLipschitz-equivalent to
  the \Kor metric, and the Lipschitz constant is locally continuous in $x$; explicit formulas can be found in \cite[p. 54]{Capo}, but let us provide our own formulas for reader's convenience.

  Computations are made easier by replacing $q$ with the orthogonally equivalent $q'(x)=2 \mathrm{Re}(\overline{x_0} x_2) - |x_1|^2$ ($x = (x_0,x_1,x_2)$). Fix, in these new coordinates,
   $x=(1,0,0)$, $x'=(0,0,1)$ in $\SB^3$. Explicitly, if we denote by $(e_0,e_1,e_2)$ the canonical basis of $\C^3$, in which $q$ is given by
    $q(x)=|x_0|^2-|x_1|^2-|x_2|^2$, 
   and we let $f_0=\frac{e_0+e_2}{\sqrt{2}}$, $f_1=e_1$, $f_2=\frac{e_0-e_2}{\sqrt{2}}$, then 
   \[ q(x_0,x_1,x_2)=q'(x_0 f_0+x_1 f_1+ x_2 f_2)\,. \]
   Note that this change of basis is orthogonal with respect to the inner product structure on $\C^3$.

  It can be checked that the Heisenberg group associated with $x$ (\emph{i.e.} stabilizing $x$) consists of the matrices of the form 
\[ \left( \begin{array}{ccc} 1 & \alpha & is + \frac{|\alpha|^2}{2} \\ 0 & 1 & \overline{\alpha} \\ 0 & 0 & 1 \end{array}\right)  \] 
  where $\alpha \in \C$ and $s \in \R$ (see \cite{LNQuint}). The orbit of $x'$ through $\Heisen$ is equal to $\SB^3 \setminus \{ x \}$ and the inverse of the Heisenberg stereographic mapping 
  is given by 
  \[ \phi:  (\alpha,s) \mapsto \left( is+\frac{|\alpha|^2}{2}, \overline{\alpha},1 \right) \in \SB^3 \]

  If we let $h = (\alpha,s)$, $h'=(\beta,t)$ be elements of $\Heisen$, a routine computation shows that the quotient 
  \[ \frac{d(\phi(h),\phi(h'))}{d(h,h')} \]
  is a continuous mapping that is uniformly bounded away from $0$ and $+\infty$ in any compact subset of $\Heisen$.

  If $K$ is a compact subset of $\SB^3\setminus \{x\}$, the restriction of $\phi$ to $K$, composed with the quotient mapping $\Heisen \to \Heisen/Z$, gives, by passing to the quotient, 
  a biLipschitz mapping 
  \[ K/R \to \Heisen/Z \]
where $K/R$ is the quotient of $K$ by the equivalence relation $R$ defined by ``$y,y'$ are equivalent if they lie on the same chain through $x$'', endowed with the quotient metric.
  
The Proposition follows.
\end{proof}

\begin{lemma} \label{l.lebdist}
  Fix $x \in \SB^3$ and let $K$ be a compact subset of $\SB^3 \setminus \{x\}$. Denote by $\HH_1$ the restriction of Lebesgue measure to $K$
  and by $\widetilde{\HH}^x$ the Borel measure on $K$ defined, for any Borel subset $A \subset K$, by 
  \[ \widetilde{\HH}^x (A) = \int \eta_{\pi_x^{-1}(u)}(A)\mathrm{d} (\pi_x \HH_1) (u)\  \]
  where $\eta_L$ is the $2$-dimensional Hausdorff measure on the chain $L$; in other words, $\widetilde{\HH}^x$ is the measure obtained by taking the Lebesgue measure on $K$ and replacing the conditional measures 
  on the fibres of $\pi_x$ with the $2$-dimensional Hausdorff measure restricted to these fibres.
  
  Then, $\tilde \HH^x$ is equivalent to $\HH_1$, and the Radon-Nikodym derivative 
  lies between $C^{-1}$ and $C$, where $C$ is a non-zero constant (depending on $K$).
\end{lemma}
\begin{proof}
This Lemma follows from the previous Proposition recalling that on the vertical lines of $\Heisen$, the conditional measures are equal to the two-dimensional Hausdorff measure. 
\end{proof}

\begin{lemma}\label{l.reg-sph}
The Lebesgue measure $\mathcal{H}$ on $\SB^3$ can be rescaled in such a way that for any $x' \in \SB^3$, 
\[ \lim_{r \to 0} \frac{\mathcal H(B(x',r))}{r^4} = 1 \]
where $B(x',r)$ is the ball of radius $r$ centered at $x'$ with respect to the visual metric $d$.
\end{lemma}
\begin{proof}
It is enough to show this for a fixed $x'$ because the group of Euclidean isometries of $\SB^3$ preserve $\mathcal H$ as well as $d$. 
Let $x,x'$ be as in the proof of Proposition \ref{p.stereo}, and let $\phi : \Heisen \to \SB^3 \setminus \{x\}$. The push-forward, through $\phi$, 
of the \Kor metric on $\Heisen$ is 
called the Hamenst\"{a}dt metric based at $x$, denoted $d_x$ (it is a metric on $\SB^3 \setminus \{x\}$); an easy computation shows that
\[ \lim_{y \to x'} \frac{d(x',y)}{d_x(x',y)} \]
exists and may be taken to be $1$ up to rescaling $d_x$. Now let also $\mathcal H_x$ be the push-forward, through $\phi$, of the Lebesgue measure $\HH$ 
on $\Heisen$, so that any ball of radius $r$ with respect to $d_x$ has $\mathcal H_x$-measure $r^4$. Existence of the previous limit then implies the exisence of 
\[ \lim_{r \to 0} \frac{\mathcal H_x (B(x',r))}{r^4} \]
and because $\mathcal H_x$ is equivalent to $\mathcal H$ and the Radon-Nikodym derivative is continuous, we conclude that
\[ \lim_{r \to 0} \frac{\mathcal H (B(x',r))}{r^4} \]
exists.
\end{proof}

\subsection{Proof of the main result}\label{ss.mainproof}
We now set out to prove Theorem \ref{th.mainres}. Fix a countable family $(B_n)$ of  balls such that any $x \in \SB^3$ belongs to infinitely many of the $B_n$, and 
\[ \inf \{ \diam B_n\ ;\ x \in B_n \} = 0\,. \]
We denote by $2 B_n$ the ball with same centre as $B_n$ and twice the radius; the radii are chosen so that $2 B_n \setminus B_n \neq \varnothing$. The closure of the complement 
$\SB^3 \setminus 2B_n$ will be denoted by $V_n$. We will work locally by using the fact that for any $x \in \SB^3$, any finite Borel measure $\mu$ giving zero measure to $\{ x \}$ can be written as
\[ \mu = \sum \mu_i \]
where, letting $(B_{n_i})$ be the family of those $B_n$ that contain $x$, $\mu_i$ is supported on $V_{n_i}$.

As in Section \ref{ss.mainres}, consider a random Poisson cut-out set $E$ and let $\mu$ be the corresponding cut-out measure supported on $E$; and fix a ball $B_{n_0}$ from the previous family. We first state the main technical lemma. Its proof is postponed to Section \ref{ss.tech}. Recall that for any chain $L$, $\eta_L$ is the $2$-dimensional Hausdorff measure restricted to $L$.

\begin{lemma} \label{l.technical}
  The space of chains $\chsp$ can be covered by open subsets $\mathcal U$ satisfying the following property: for any $L,L' \in \mathcal U$,
  \begin{equation}
    \left| \eta_L \left( \bigcup_{i=1}^N B_i \right) - \eta_{L'} \left( \bigcup_{i=1}^N B_i \right) \right| \lesssim N \cdot d(L,L')^{1/4}
  \end{equation}
  for any finite family of (visual) balls $(B_i)_{1 \leq i \leq N}$. 
The constant implied in the notation $\lesssim$ depends only on $\mathcal U$.
\end{lemma}

Assuming Lemma \ref{l.technical} holds, we will fix $n_0\in\N$ and prove the statement of Theorem \ref{th.mainres} for $\mu|V_{n_0}$ and $\pi_x$, $x\in B_{n_0}$.

\begin{lemma}\label{l.main-lemma}
  Conditional on $\mu(V_{n_0}) \neq 0$, the following, where $\mu' = \mu | V_{n_0}$, holds almost surely:  for any $x \in B_{n_0}$, 
  \[ \dim (\pi_x (\mu')) = \inf \{ 2 , \dim (\mu') \}\,. \]
  Moreover, $\pi_x(\mu')$ is absolutely continuous and $\pi_x(E\cap V_{n_0})$ has non-empty interior, if $\dim(\mu')>2$.
\end{lemma}
\begin{proof}[Proof of the Lemma]
  Let $\mathcal K$ be the space of all chains passing through the compact subsets $B_{n_0}$ and $V_{n_0}$. Then, $\mathcal{K}$ is compact; indeed 
  the mapping that sends a pair $(x,y)$ of distinct points of $\SB^3$ to the chain passing through $x$ and $y$ is continuous, thus compactness 
  of $\mathcal K$ follows from the compactness of $B_{n_0} \times V_{n_0}$.

  By virtue of 
  Lemma \ref{l.technical}, we can cover $\mathcal K$ with open sets $\mathcal{U}_1,\ldots,\mathcal{U}_p$, such that the conclusion of Lemma \ref{l.technical} holds for any $L,L'\in\mathcal U_i$. Denote   $\mathcal K_i=\mathcal K\cap\mathcal U_i$.
  
  For each index $i$, we wish to apply Theorem \ref{th:holdcont} to 
  \begin{itemize}
    \item the restricted SI-martingale $(\mu_n')_n$ where $\mu_n' =\mu_n|V_{n_0}$ (this is again an SI-martingale);
    \item the space of chains $\Gamma=\mathcal K_i$;
    \item the family of measures $(\eta_L)_{L \in \mathcal K_i}$ where for any chain $L \in \mathcal K_i$, $\eta_L$ is the 
    $2$-dimensional Hausdorff measure on $L$: $\eta_L = H^2|L$. 
  \end{itemize}
  
 We will also apply Lemmas \ref{l.supcrit} and \ref{l.subcrit} for the projections $\pi_x$, $x\in B_{n_0}$. Note that the co-domain of $\pi_x$, 
 $x^\perp\setminus\{x\}\subset\pdc$, is a punctured Euclidean $2$-sphere, which is locally biLipschitz equivalent to $\R^2$. Thus we may apply these lemmas for $k=2$.
  
Let us now check that the assumptions of Theorem \ref{th:holdcont} are satisfied.  Assumption \eqref{H:hyp1} holds trivially. 
From \eqref{eq:mu_n} it follows that  \eqref{H:hyp3} holds with any exponent  $\gamma'>\gamma$. 
Assumption \eqref{H:hyp2} is the content of Lemma \ref{haus-two-frostman}.
  
  Finally, to verify the assumption \eqref{H:hyp4}, we note that if $N_n$ denotes the number of Poisson cut-out balls with radius $>2^{-n}$ (i.e. those $(x_i,r_i)\in\mathcal{Y}$ for which, $r_i>2^{-n}$), then almost surely, there is a random integer $M_0$ such that $N_n\le 2^{5n}$ for all $n\ge M_0$. See \cite[Lemma 5.15]{ShmerkinSuomala} for a proof of this fact. (Here $5$ may be replaced by any number $>4$). Combining this with Lemma \ref{l.technical} yields
  \begin{align*}
  \int\mu_n\,\mathrm{d}\eta_{L}-\int\mu_n\,\mathrm{d}\eta_{L'}\le C 2^{n(5+\gamma')}d(L,L')^{1/4}
  \end{align*}
  for any $n\ge N_0$ and for all $L,L'\in\mathcal{K}_i$, recall \eqref{eq:mu_n}.
  
Thus, the assumptions of Theorem \ref{th:holdcont} are satisfied.  Let us now consider the case $\dim (\mu') > 2$ (this is the case when, in the notations of Theorem \ref{th:holdcont}, $\gamma < 2$, since $\dim(\mu') = 4-\gamma$). Theorem \ref{th:holdcont} implies that for any $L \in \mathcal K_i$, 
  \[ \int \mu'_n \ \mathrm{d} \eta_L \]
  converges uniformly to a finite number $X(L)$ and the mapping $L \mapsto X(L)$ is continuous on $\mathcal{K}_i$. Since the sets $\mathcal{K}_i$ are relatively open, this mapping remains continuous on $\mathcal{K}=\mathcal{ K}_1\cup\ldots\cup\mathcal{K}_p$ as well. Now fix some $x \in B_{n_0}$ and
  apply Lemma \ref{l.supcrit} to the compact metric space $\mathcal Z = V_{n_0}$, the projection $\pi = \pi_x$, 
  the measure $\HH | V_{n_0}$, the sequence of Borel functions $(\mu_n'|V_{n_0})_n$ and the family of fibre measures 
  $(\eta_L|V_{n_0})$ where $L$ goes through all chains passing through $x$ and meeting $V_{n_0}$. This Lemma 
  yields the absolute continuity of $\pi_x \mu'$, and the fact that $\pi_x (\supp \mu')$ has non empty interior, as desired.

  Now we look at the case $\dim (\mu') \leq 2$ and fix some $\theta > 2 - \dim(\mu')$. The conclusion of Theorem \ref{th:holdcont} now gives, for 
  any chain $L \in \mathcal K$, and any $n$,
\begin{equation}\label{eq.main-proof} 
\int \mu_n'\ \mathrm{d} \eta_L \lesssim 2^{\theta n}\,. 
\end{equation}

In order to apply Lemma \ref{l.subcrit}, we still need to check that, almost surely, the assumption \eqref{eq:uniform_tube_bound} in that lemma holds simultaneously for 
each $\pi_x$, $x\in B_{n_0}$. 

Let us fix $\varepsilon=1/(1000 C)$, where $C$ is the constant from Lemma \ref{technical-lemma}, when the lemma is applied for $B=B_{n_0}$, $V=V_{n_0}$. 
For each $n$, let $\mathcal{D}_n$ be an $(\varepsilon2^{-n})$- dense subset of $\mathcal{K}$ and for each $L\in\mathcal{D}_n$ pick $x\in B_{n_0}$ and $u\in \pi_x(V_0)$ such that $L=\pi^{-1}_x(u)$ and consider
\[T_L:=\pi_x^{-1}(B(u,2^{1-n}))\cap V_{n_0}\,.\]
Note that such a $\mathcal{D}_n$ may be chosen to have cardinality $\le C_{\mathcal{K},\varepsilon}2^{4n}$.

Using Lemma \ref{mu_m_to_mu_c.} as in the proof of Theorem \ref{p.supcrit} implies the existence of a random constant $M<+\infty$ such that
\begin{equation}\label{mu_n_c_mu}
\mu'(T_L)\le M(\mu'_n(T_L)+2^{n(\theta-2)})
\end{equation}
for all $L\in\mathcal{D}_n$ and all $n\in\N$.

Now, let us fix $x\in B_{n_0}$ and let $n\in\N$. For each $L=L_u\in\mathcal{D}_n$, consider $u_{L,x}\in\pi_x(V_{n_0})$ such that $d_E(u_{L,x},u)\le \varepsilon 2^{-n}$ if there is any. Let $\mathcal{D}_n^x$ be the collection of all such $u_{L,x}$. It follows from Lemma \ref{technical-lemma} that
$\mathcal{D}^x_n\subset \pi_x(V_{n_0})$ is $2^{-n}$-dense and 
\begin{equation}\label{eq:tube_incl}
\pi_x^{-1}(B(u',2^{-n}))\cap V_{n_0}\subset T_L\subset \pi_x^{-1}(B(u',C 2^{-n}))\,,
\end{equation}
whenever $u'\in \mathcal{D}_n^x$ is such that $u'=u_{L,x}$.

Combining \eqref{mu_n_c_mu} and \eqref{eq:tube_incl} we have
\begin{align*}
\pi_x\mu'(B(u,2^{-n}))\le C\left(\mu_n'(B(u,C2^{-n}))+2^{n(\theta-2)}\right)\,,
\end{align*} 
for all $n\in\N$, and all $u\in\mathcal{D}_n^x$
Recalling \eqref{eq.main-proof}, we may now apply Lemma \ref{l.subcrit}
 
  which implies that $\dim (\pi_x \mu') \geq 2-\theta$. Note that, almost surely, this holds for all $x\in B_{n_0}$ simultaneously.  Hence 
  the conclusion.
\end{proof}

\begin{proof}[Proof of Theorem \ref{th.mainres}]
If $\mu \neq 0$, for any $x \in \SB^3$ we can write $\mu$ as a countable sum
\[ \mu = \sum_i \mu^i \]
where each $\mu^i\neq 0$ is supported on some $V_{n_i}$ and $x$ belongs to the corresponding $B_{n_i}$. Now for any $i$, 
$\mu^i$ has same dimension as $\mu$, and $\pi_x (\mu^i)$ is absolutely continuous, resp. has same dimension as $\mu^i$, if 
$\dim (\mu) > 2$, resp $\dim (\mu) \leq 2$. The same must hold for $\pi_x (\mu) = \sum_i \pi_x (\mu^i)$. In the same way, one obtains that 
$\pi_x (\supp \mu)$ is non-empty if $\dim(\mu) > 2$. 
\end{proof}

\subsection{Technical Lemma}\label{ss.tech}

It remains to prove the technical Lemma \ref{l.technical}. We will accomplish this in several parts. One of the key steps is an estimate on the size of the intersection of a chain and an annulus, see Lemma \ref{l.intersec}. Recall that $G=\mathbf{SU}(1,2)$.

\begin{lemma}\label{l.intersec}
  Let $\mathcal K$ be a compact subset of the space of chains $\chsp$.
  There is a constant $r_0$ such that for any $x \in \SB^3$, $L \in \mathcal K$ and $0 < \delta \leq r \leq r_0$,
  \begin{equation}
    \eta_L \left( A(x,r,\delta) \right) \lesssim \delta^{1/2}
  \end{equation}
  where $\eta_L$ is the $2$-dimensional Hausdorff measure on $L$, and 
  \[ A(x,r,\delta) = \{ y \in \SB^3\ ;\ r \leq d(x,y) \leq r+\delta \}\,. \]
\end{lemma}

The proof of this lemma relies on two facts. 
First, we devise an explicit parametrization of chains in general position. 

Fix $w = [1:w_1:w_2]$ where $|w_1|^2+|w_2|^2 > 1$, so
that $w$ belongs to $\chsp$, and let $\kappa^2 = |w_1|^2 + |w_2|^2$ and $\upsilon^2=\kappa^{2}-1$. We denote by $L$ the chain $w^\perp \cap \SB^3$.
 Then the mapping
\begin{equation}\label{eq.y-theta}
  \theta \mapsto  y_\theta = [\kappa^2,y_1,y_2] \in \pdc \quad ; \quad  (y_1,y_2) = (w_1,w_2)+ \upsilon e^{i \theta} (-\overline{w_2},\overline{w_1})
\end{equation}
(where $\theta \in [0,2\pi[$) is a smooth parametrization of $L$. If $1+\varepsilon \leq \kappa^2 \leq \varepsilon^{-1}$,
the modulus of the derivative of this mapping is bounded away from $0$ and $+\infty$ by a constant depending only on $\varepsilon$; in particular, $\theta \mapsto y_\theta$
is, by \eqref{i.grom1} in Lemma \ref{lemma.gromov}, $\frac{1}{2}$-H\"{o}lder (with a multiplicative constant depending on $\varepsilon$) when $[0,2\pi[=\R/(2 \pi \mathbb{Z})$ is endowed with
the usual torus metric and $\SB^3$ is endowed with the Heisenberg metric.

Secondly, we need an elementary estimate from plane geometry. Let $\SB^1$ be the unit circle in the complex plane. For any $0 < \delta \leq r \leq 0.001$ (say),
and any $z \in \C$, the Euclidean length of the intersection of $\SB^1$ with the annulus $A(z,r,\delta) = \{ u \in \C\ ;\ r \leq d(z,u) \leq r+\delta\}$ is
dominated by $\delta^{1/2}$, \emph{i.e.}
\begin{equation}\label{eq.plest} H^1 (\SB^1 \cap A(z,r,\delta)) \lesssim \delta^{1/2}\,. \end{equation} 

  We leave it to the reader to verify the claims of the last paragraphs. Let us now prove the Lemma \ref{l.intersec}.
  \begin{proof}[Proof of Lemma \ref{l.intersec}]
    Our approach is fairly down-to-earth: we prove the needed estimate for a fixed $x$ which allows for explicit computations, and we use the transitivity of
    $K=\mathbf{SO}(3)$ on $\SB^3$ to deduce that the Lemma holds
    for any $x \in \SB^3$. Until further notice, we let $x$ be the fixed element $[1:1:0]$ of $\SB^3$.
    
    We denote by $\mathcal L_0$ the set of all $w=[1:w_1:w_2] \in \pdc$ where $\kappa^2 = |w_1|^2+|w_2|^2 > 1$ and $w_2 \neq 0$, and $\mathcal L_0(\varepsilon)$ those
    $w$ such that also $1+\varepsilon \leq \kappa^2 \leq \varepsilon^{-1}$ and $|w_2| \geq \varepsilon$. Any compact subset of $\mathcal L_0$ is contained in some
    $\mathcal L_0(\varepsilon)$ for $\varepsilon$ small enough. Now fix $\varepsilon >0$ and $w \in \mathcal L_0(\varepsilon)$,
    and let $\theta \mapsto y_\theta$ be
    the mapping onto the chain $L=w^\perp \cap \SB^3$ defined above \eqref{eq.y-theta}.
    
    A simple computation gives
    \begin{equation}
      d(x,y_\theta)^2 = \frac{|\kappa^2 - y_1|}{2 \kappa^2} = \frac{|w_2| \upsilon}{2 \kappa^2} |e^{i\theta} - z|\,,
    \end{equation}
    where we denote
    \[ z = \frac{w_1-\kappa^2}{\overline{w_2} \upsilon}\,. \]
    Direct application of \eqref{eq.plest} yields that there is a constant $r_0 = r_0(\varepsilon)$ such that for $0 < r < r_0(\varepsilon)$
    and $ 0 < \delta \leq r^2$,
    \[ \mathrm{H}^1 (\{ \theta \in [0,2\pi[\ ;\ r \leq d(x,y_\theta) \leq r+\delta \}) \lesssim_\varepsilon \delta^{1/2} \]
    If, on the other hand, $r^2 < \delta \leq r$, then it holds trivially that
    \[ \mathrm{H}^1 (\{ \theta \in [0,2\pi[\ ;\ r \leq d(x,y_\theta) \leq r+\delta \}) \lesssim r \leq \delta^{1/2} \]
    
    We thus see that this estimate holds, provided that $0 < \delta \leq r \leq r_0(\varepsilon)$, and we deduce (using the $\tfrac12$-H\"{o}lderness of
    $\theta \mapsto y_\theta$) that
  \begin{equation}\label{eq.prest} \eta_L \left( A(x,r,\delta) \right) \lesssim_\varepsilon \delta^{1/2} \end{equation}
    for $0 < \delta \leq r \leq r_0(\varepsilon)$, and $w \in \mathcal L_0(\varepsilon)$.
    
    Let us now denote by $\mathcal L_1$ the set of all $w=[1:w_1:0]$ where $|w_1| >1$, and let $\mathcal L_1(\varepsilon)$ be those $w$ such that
    $1+ \varepsilon \leq |w_1|^2 \leq \varepsilon^{-1}$. Since
\[\frac{|\langle x,\omega\rangle|}{||x|| ||\omega||}\gtrsim 1-|\omega_1|^2\ge\varepsilon\,,\]    
    there is a constant $r_1(\varepsilon)$ such that for any $w \in \mathcal L_0(\varepsilon)$, and any
    $0 < \delta \leq r \leq r_1(\varepsilon)$, the intersection $A(x,r,\delta) \cap L$
     is empty;
    the estimate \eqref{eq.prest} thus holds trivially also in this case.
    
    Finally, we leave it to the reader to deal with the subset $\mathcal L_2\subset\mathcal{L}_\C$ of all $w=[0:w_1:w_2]$ where $w_1$ and $w_2$ are not both $0$.
    
    All in all, we have the following: let  $\mathcal K$ be a compact subset of $\chsp$; there is a constant $r(\mathcal K)$
     such that, for any $L \in \mathcal K$ and $0 < \delta \leq r \leq r(\mathcal K)$,
  \begin{equation} \eta_L \left(  A(x,r,\delta) \right) \lesssim \delta^{1/2}\,. \end{equation}
    
    Now recall the Iwasawa decomposition $\mathbf{PU}(1,2)=KAN$, where the operation of $K$ on $\SB^3 \subset \pdc$ identifies with the natural operation of
    $\mathbf{SO}(3)$. This operation is transitive, and it preserves both $d_E$ and $d$, as well as chains (in other words, the image of a chain
    through an element of $K$ is another chain).
    
    Apply the result above to the compact subset $K \mathcal K \subset \pdc$ instead in $\mathcal K$. We obtain $r(K \mathcal K)>0$ such
    that, for any 
    $0<\delta\le r\le r(K\mathcal K)$, and any $w \in K \mathcal K$, 
     \[ \eta_{L_w} (A(x,r,\delta))  \lesssim \delta^{1/2}\,, \]
(where still $x=[1:1:0]$). 
If $\omega\in\mathcal K$ and $x'\in\SB^3$, let $g\in K$ such that $gx'=x$. Then  
\begin{align*} &\eta_{L_w} (A(x',r,\delta))
= H^2_\Heis(w^\perp \cap A(x',r,\delta)) = H^2_\Heis((gw)^\perp \cap A(x,r,\delta)) \\
&= \eta_{L_{gw}} (A(x,r,\delta))\lesssim\delta^{1/2}\,, 
\end{align*}
as desired.
  \end{proof}
  
We may now finally complete the proof of Lemma \ref{l.technical}.  
\begin{proof}[Proof of Lemma \ref{l.technical}]
   Let $L_1,L_2$ be two chains and let $g \in G$ be such that $L_2 = g L_1$. For any Borel subset $A$ of $S^3$,
  \begin{equation}\label{eq.techn1}
    \eta_1 (A) - \eta_2 (A) \leq O\left( \| \Id - g \| \right) + \eta_2 (gA \setminus A)
  \end{equation}
  Indeed, $\eta_1(A) = g_* \eta_1 (gA)$ (where $g_* \eta_1$ is the push-forward of $\eta_1$ through $g$) and $g^{-1} : L_2 \to L_1$ is a Lipschitz mapping
  with Lipschitz constant $(1+\|\Id - g\|)$ (see \eqref{eq.gromlip}), so
  \begin{equation}
    g_* \eta_1 (gA) \leq \left( 1+\|\Id - g\| \right)^2 \eta_2 (gA)
  \end{equation}
  and \eqref{eq.techn1} follows.
  
  In particular, if $A$ is the union $\cup_i B_i$ of $N$ balls, \eqref{eq.techn1} implies
  \begin{align*}
    &\left| \eta_1 \left( \bigcup_i B_i \right) - \eta_2 \left( \bigcup_i B_i \right) \right| \\
    &\leq O\left(\|\Id-g\| + \|\Id - g^{-1} \|\right) +
    \sum_i \eta_1 (g^{-1} B_i \setminus B_i) + \eta_2 (g B_i \setminus B_i)\,.
  \end{align*}
  
  In view of this, our task is to show that locally we can find $g$ such that $d(L_1,L_2) \asymp \| \Id-g\| \asymp \| \Id - g^{-1} \|$
  and to bound $\eta (gB \setminus B)$ by (a power of) $\| \Id-g \|$ where $\eta$ is the $2$-dimensional Hausdorff measure on some chain $L$ sitting
  in a compact subset of $\chsp$.
  
  The first step is accomplished in Lemma \ref{subl1}. The second step follows from Lemma \ref{l.intersec}, and from
  the fact that if $B=B(x,r)$ is a ball of radius $r$, then by \eqref{eq.balldist} 
  \[gB \setminus B\subset A\left(x,r,\sqrt{\|\Id-g\| \cdot \|g^{-1} \|}\right)\,,\] where $A(x,r,\delta) = \{ y\in \SB^3\ ;\ r \leq d(x,y) \leq r+\delta \}$.
\end{proof}

  \bibliographystyle{plain}
  \bibliography{bibli}

\end{document}